\documentclass{article}
\usepackage{graphicx} 
\usepackage{R/mypreamble}

\usepackage[backend=biber, bibstyle=R/mybibstyle, citestyle=ieee-alphabetic, date=year, giveninits=true, dashed=true]{biblatex}
\usepackage{R/mypaper}

\DeclareMathOperator{\mdeg}{mdeg}
\DeclareMathOperator{\Iso}{Iso}
\DeclareMathOperator{\PIso}{PIso}

\title{\textbf{\uppercase{\large{Automorphism groups of deformations and quantizations of Kleinian singularities}}}}

\author{Simone Castellan}
\date{}

\begin{document}

\maketitle

\begin{abstract}
    \noindent It is known that, for the algebra of functions on a Kleinian singularity, the parameter space of deformations and the parameter space of quantizations coincide. We prove that, for a Kleinian singularity of type $\mathbf{A}$ or $\mathbf{D}$, isomorphisms between the quantizations are essentially the same as Poisson isomorphisms between deformations. In particular, the group of automorphisms of the deformation and the quantization corresponding to the same deformation parameter are isomorphic. We additionally describe the groups of automorphisms as abstract groups: for type $\mathbf{A}$ they have an amalgamated free product structure, for type $\mathbf{D}$ they are subgroups of the groups of Dynkin diagram automorphisms. For type $\mathbf{D}$ we also compute all the possible affine isomorphisms between deformations; this was not known before.
\end{abstract}

\section{Introduction}
\renewcommand*{\thetheorem}{\Alph{theorem}}
Consider a non-commutative, associative filtered algebra $B$ and its semi-classical limit $\Bb$, which is a Poisson algebra. The Poisson structure on $\Bb$ can be thought of as a first-order approximation of the algebra structure on $B$. It is then natural to ask how much does $\Bb$ recover of the properties of $B$. This question has a rich history; there is no general theory, but a lot of work has been done in this direction for specific algebras. In \cite{goodearlSemiclassicalLimitsQuantized2010}, it is conjectured that the prime and primitive spectra of the quantized coordinate rings are respectively homeomorphic to the Poisson prime and Poisson primitive spectra of their corresponding semi-classical limits when the base field is algebraically closed and of characteristic zero. This conjecture has since been verified in a number of examples \cite{fryerPrimeSpectrumQuantum2017, goodearlSemiclassicalLimitsQuantum2009, goodearlSemiclassicalLimitsQuantized2010}. In \cite{launoisPoissonDerivationsSemiclassical2023, launoisDerivationsFamilyQuantum2023}, the authors studied the derivations of certain deformations $A_{\a,\b}$ of the second Weyl algebra and the Poisson derivations of their semi-classical limit $\Aa_{\a,\b}$, proving that the first Hochschild cohomology group $HH^1(A_{\a,\b})$ is isomorphic to the first Poisson cohomology group $HP^1(\Aa_{\a,\b})$. In \cite{choSemiclassicalLimitsOre2016}, it is proved that the endomorphisms of the generalized Weyl algebras are the same as the Poisson endomorphisms of the Poisson generalized Weyl algebras. The connection between the $0$-th Poisson homology $HP_0(\Bb)$ and the $0$-th Hochschild homology $HH_0(B)$ was studied by Etingof and Schedler in a series of papers \cite{etingofPoissonTracesDModules2010, etingofTracesFiniteMathcal2010, etingofZerothPoissonHomology2012, etingofPoissonTracesSymmetric2014}, proving that $HP_0(\Bb)\cong HH_0(B)$ in many examples, including symmetric powers of isolated quasi-homogeneous surface singularities and finite $W$-algebras.

A particularly interesting problem concerns the relation between automorphisms. Belov-Kanel and Kontsevich made the following remarkable conjecture: 
\begin{conj}[{\cite{belov-kanelAutomorphismsWeylAlgebra2005}}]\label{conj kontsevich}
The automorphism group of the Weyl algebra of index $n$ over
$\CC$ is isomorphic to the group of Poisson automorphisms of $\CC[x_1,\dots,x_n,y_1,\dots,y_{2n}]$, with the standard symplectic structure.
\end{conj}
The conjecture is verified for $n=1$ and for the subgroups of tame automorphisms \cite{belov-kanelAutomorphismsWeylAlgebra2005}. The importance of Conjecture \ref{conj kontsevich} relies on its connection to other important open problems, like the Jacobian conjecture and the Dixmier conjecture \cite{belov-kanelJacobianConjectureStably2007,belov-kanelPolynomialAutomorphismsQuantization2022}. It is a hard problem, and it motivates the study of the relation between the automorphism groups of Poisson algebras and their quantizations in other cases. 

For a graded Poisson algebra $\Bb$, a filtered deformation is a filtered Poisson algebra $\Cc$ such that $\gr(\Cc)\cong\Bb$ as graded Poisson algebras. Suppose that the deformation and quantization theory of $\Bb$ has the desirable property that deformations and quantizations are parameterized by the same moduli space $\mathfrak{M}$ (see \cite{ambrosioUniversalFilteredQuantizations2022}). We denote by $B_{0,c}$ and $B_{1,c}$ respectively the deformation and quantization corresponding to the parameter $c\in\mathfrak{M}$. Consider all possible associative algebra isomorphisms between the various quantizations and all Poisson isomorphisms between the various deformations. They form two groupoids, where the objects are the algebras $B_{1,c}$ (respectively $B_{0,c}$) and the morphisms are the (Poisson) isomorphisms; we denote them by $\Iso(B)$ (respectively $\PIso(B)$). In this paper, we shall investigate the relation between $\Iso(B)$ and $\PIso(B)$, through the case study of the algebra of functions on a Kleinian singularity. In particular, understanding the groupoid of isomorphisms will give us information about the automorphism groups of all the deformations and quantizations.

Recall that a Kleinian singularity is a quotient space $V/\Gamma$, where $V$ is a two dimensional complex vector space, and $\Gamma$ is a finite subgroup of $SL(2,\CC)$. Kleinian singularities are isomorphic to surfaces in $\AA^3$, so
\begin{equation*}
    \CC[V/\Gamma]=\CC[V]^\Gamma\cong\CC[x,y,z]/(F), \quad \text{with $F\in\CC[x,y,z]$.}
\end{equation*}
Moreover, $\CC[V]^\Gamma$ carries the structure of a Poisson algebra, induced by the standard symplectic form on $V$. Kleinian singularities, and their corresponding algebras of functions, are classified by simply laced Dynkin diagrams (type $\mathbf{A_n} \ n\geq1, \mathbf{D_n} \ n\geq4, \mathbf{E_6},\mathbf{E_7},\mathbf{E_8}$) (see \cite{mckayGraphsSingularitiesFinite1980}). 

The study of deformations and quantizations of Kleinian singularities has a long and rich history. In both cases the moduli space is $\hh/W$, where $\hh$ and $W$ are, respectively, the Cartan subalgebra and the Weyl group of the Lie algebra of corresponding Dynkin type. The construction of the semi-universal deformation is due to the work of Grothendieck, Brieskorn, Kronheimer and Slodowy \cite{slodowySimpleSingularitiesSimple1980, kronheimerConstructionALESpaces1989, katzGorensteinThreefoldSingularities1992}. The quantizations of a Kleinian singularity were constructed by Crawley-Boevey and Holland in \cite{crawley-boeveyNoncommutativeDeformationsKleinian1998} as a family of algebras that they denote $\Oo^\l$. In \cite{crawley-boeveyNoncommutativeDeformationsKleinian1998} it is also proved that $\Oo^\l$ is isomorphic to $e_0\Pi^\l(Q) e_0$, the spherical subalgebra of the deformed preprojective algebra of the McKay quiver $Q$ associated to the singularity (see \cite{mckayGraphsSingularitiesFinite1980}). Explicit presentation with generators and relations for the quantizations in type $\mathbf{A}$ appeared in earlier works by Hodges \cite{hodgesNoncommutativeDeformationsTypeA1993}, Smith \cite{smithClassAlgebrasSimilar1990}, and  Bavula and Jordan, who studied them as special cases of generalized Weyl algebras \cite{bavulaGeneralizedWeylAlgebras1992,bavulaIsomorphismProblemsGroups2001}. For type $\mathbf{D}$, an explicit construction of the quantizations was given by Levy in \cite{levyIsomorphismProblemsNoncommutative2009} and, using different methods, by Boddington \cite{boddingtonDeformationsTypeKleinian2007}. 

We formulate the following conjecture.
\begin{conj}\label{conj main}
    Let $V/\Gamma$ be a Kleinian singularity. Then there is an isomorphism of groupoids
    \begin{equation*}
        \PIso(\CC[V/\Gamma])\cong\Iso(\CC[V/\Gamma]).
    \end{equation*}
\end{conj}

Since the only filtered quantization of the polynomial algebra $\CC[x_1,\dots,x_n,y_1,\dots,y_{2n}]$ is the Weyl algebra, Conjecture \ref{conj main} can be seen as a direct analog of Belov-Kanel-Kontsevich Conjecture for the algebra of function on a Kleinian singularity. The main result of this paper is the following theorem (see Theorems \ref{teo iso type A} and \ref{teo iso type D})

\begin{theorem}\label{teo main}
    Let $V/\Gamma$ be a Kleinian singularity of type $\mathbf{A}$ or $\mathbf{D}$. Then Conjecture \ref{conj main} holds. In particular, the group of automorphisms of the quantization and of Poisson automorphisms of the deformation corresponding to the same parameter are isomorphic.
\end{theorem}

The isomorphism problem for the quantizations was already solved in \cite{bavulaIsomorphismProblemsGroups2001}, for type $\mathbf{A}$, and in \cite{levyIsomorphismProblemsNoncommutative2009}, for type $\mathbf{D}$, while the Poisson side is less studied in the literature. In type $\mathbf{A}$, the automorphism groups of the deformations (considered as affine varieties) were computed by Makar-Limanov \cite{makar-limanovGroupsAutomorphismsClass1990} and further studied by Blanc \cite{blancAutomorphismsPreFibered2011}. Naurazbekova and Umirbaev studied in \cite{naurazbekovaAutomorphismsSimpleQuotients2021} the groups of Poisson automorphisms of deformed Kleinian singularities in type $\mathbf{A_1}$. They proved that the groups are independent of the deformation parameter and that they have an amalgamated free product structure. In the case $n=1$, the deformations and quantizations can be identified with quotients of the symmetric algebra and universal enveloping algebra of $\ss\ll_2$, respectively. The automorphism groups of these quotients of $U(\ss\ll_2)$ are known to have the same amalgamated free product structure \cite{dixmierQuotientsSimplesAlgebre1973, fleurySousgroupesFinisAut1998}, which implies that the two groups are isomorphic \cite[Theorem 5]{naurazbekovaAutomorphismsSimpleQuotients2021}. In the present paper, we show that these results can be generalized to $n>1$, even if we no longer have the $\ss\ll_2$ structure in higher degree. The main difference is that, for $n>1$, the (Poisson) automorphism groups depend on the deformation parameter.

\begin{theorem}
    Let $n>2$. The group $G$ of Poisson automorphisms of a deformation of a $\mathbf{A_{n-1}}$ Kleinian singularity is isomorphic to the group of automorphisms of the corresponding quantization. The dependence of $G$ on the deformation parameter splits into two classes, one for generic and one for special parameter (see Theorem \ref{teo amalgamated free product} for more details):
    \begin{enumerate}[(i)]
        \item for special deformation parameter and $n$ even,
        $G\cong (\CC[y]\rtimes\CC^\times)\ast_{\CC^\times} (\CC^\times\rtimes\ZZ/2\ZZ);$
        \item for special deformation parameter and $n$ odd,
        $G\cong (\CC[y]\rtimes\CC^\times)\ast_{\CC^\times} H$, where \\$H=\langle \CC^\times, \Omega \, | \, \Omega^2=-1, \,\l\cdot\Omega=\Omega\cdot\l^{-1} \ \forall\l\in\CC^\times \rangle$;
        \item for generic deformation parameter, $G\cong (\CC[y]\rtimes\CC^\times)\ast_{\CC^\times} (\CC[x]\rtimes\CC^\times).$
    \end{enumerate}
\end{theorem}

In type $\mathbf{D}$, the affine automorphism group of the undeformed Kleinian singularity was computed by Blanc in \cite{blancNonrationalityFibrationsAssociated2015}. To our knowledge, the affine isomorphisms between deformations were unknown; we compute them in Theorem \ref{teo iso Dn}. We are then able to check which ones are Poisson, directly confirming Conjecture \ref{conj main}. In particular, we can give an explicit presentation of the groups of Poisson automorphisms of the deformations.

\begin{theorem}
    Let $n\geq4$. The group $G$ of Poisson automorphisms of a deformation of a $\mathbf{D_{n}}$-Kleinian singularity is isomorphic to the group of automorphisms of the corresponding quantization. The dependence of $G$ on the deformation parameter splits into three classes, one for generic, one for sub-generic and one for special parameter (see Theorem \ref{teo iso quantizations type D} for more details):
    \begin{enumerate}[(i)]
        \item for generic deformation parameter, $G=\langle \id \rangle$;
        \item for $n>4$ and special deformation parameter, $G=\ZZ/2\ZZ$;
        \item for $n=4$ and special deformation parameter, $G=S_3$, the symmetric group on $3$ elements;
        \item for $n=4$ and sub-generic deformation parameter, $G=\ZZ/2\ZZ$.
    \end{enumerate}
\end{theorem}

Even though the results are uniform, the methods used to prove them are quite different: algebraic for type $\mathbf{A}$, geometric for type $\mathbf{D}$. Additionally, in type $\mathbf{A}$ we compute the automorphism groups first, and then use that to get the whole isomorphism groupoid, while in type $\mathbf{D}$ we are not able to compute the automorphism groups directly, but we need to compute the whole isomorphism groupoid first. 

Recall that a Kleinian singularity $V/\Gamma$ is a symplectic quotient, thus the results of the present paper can be seen in the more general context of symplectic singularities (in the sense of Beauville \cite{beauvilleSymplecticSingularities2000}). In \cite{namikawaPoissonDeformationsAffine2010}, Namikawa constructed a Cartan space $\mathfrak{P}$ and Weyl group $W$ associated to a conic symplectic singularity $X$, with $W$ acting on $\mathfrak{P}$ by crystallographic reflections, and proved that the deformations of $\CC[X]$ are parameterized by $\mathfrak{P}/W$. Recently, Losev proved that the same is true for the quantizations of $\CC[X]$ (see \cite[Theorem 3.4]{losevDeformationsSymplecticSingularities2022}). The reductive part of the group of graded Poisson algebra isomorphisms of $\CC[X]$ acts on the space of deformations and quantizations as filtered (Poisson) algebra isomorphisms. By a general result of Losev \cite[Proposition 3.21 and Corollary 3.22]{losevDeformationsSymplecticSingularities2022}, these two actions coincide and induce all the possible filtered (Poisson) algebra isomorphisms between different deformations or quantizations. Theorem \ref{teo main} can be restated as saying that, for the special case of type $\mathbf{A}$ and $\mathbf{D}$ Kleinian singularities, all (Poisson) isomorphisms come from the action of an automorphism of $\CC[X]$ on the moduli space. What we show also in the present paper is that in type $\mathbf{D}$ all (Poisson) isomorphisms are filtered, while in type $\mathbf{A}$ the non-filtered (Poisson) isomorphisms come from the exponentiation of nilpotent inner derivations. 

Notice that the associated graded of each of the quantizations of a Kleinian singularity $X$ is not the corresponding deformation, but the algebra $\CC[X]$ itself. Nonetheless, there is a different way to realize a deformation as the semi-classical limit of the corresponding quantization. For each symplectic quotient, Etingof and Ginzburg \cite{etingofSymplecticReflectionAlgebras2002} constructed a symplectic reflection algebra $H_{t,c}$, for all $t\in\CC$ and all $c$ in a certain vector space of dimension $\dim \mathfrak{P}$. They are defined as quotients of the smash-product $\Tt(V^*)\rtimes\Gamma$, where $\Tt(V^*)$ is the tensor algebra of $V^*$. Let $e$ be the idempotent element in $\Gamma$; the algebra $eH_{t,c}e$ is called the spherical subalgebra of $H_{t,c}$. 
A result by Bellamy \cite{bellamyCountingResolutionsSymplectic2016} shows that every filtered Poisson deformation of $\CC[V]^\Gamma$ is of the form $eH_{0,c}e$, for some $c$. In \cite{losevDeformationsSymplecticSingularities2022}, it is proved that every filtered quantization $\CC[V]^\Gamma$ is of the form $eH_{1,c}e$, for some $c$; we can thus regard every deformation as the semi-classical limit of $eH_{t,c}e$, as $t$ goes to $0$. Unfortunately, this does not seem to allow us to directly compare the automorphisms of the two objects.

The paper is structured as follows. In Section \ref{section 2} we study the case of type $\mathbf{A}$ singularities. We prove that the group of automorphisms of every quantization admits an amalgamated free product structure (Theorem \ref{teo amalgamated free product}). We compute the group of Poisson automorphisms of every deformation, using techniques similar to \cite{naurazbekovaAutomorphismsSimpleQuotients2021}, and prove that it is isomorphic to the automorphism group of the corresponding deformation (Theorem \ref{teo poisson automorphisms type A}). Using these results we are then able to confirm Conjecture \ref{conj main} (Theorem \ref{teo iso type A}) in type $\mathbf{A}$. In Section \ref{section 3} we study the case of type $\mathbf{D}$ singularities. Using techniques similar to \cite{blancNonrationalityFibrationsAssociated2015}, we manage to compute all the isomorphisms between deformations as affine varieties (Theorem \ref{teo iso Dn}); we then compute the subgroupoid of Poisson isomorphisms and confirm Conjecture \ref{conj main} (Theorem \ref{teo iso type D}).

\paragraph{Acknowledgments.} I am grateful for the constant encouragement, guidance and support provided by my supervisors Daniele Valeri and Gwyn Bellamy, which made this work possible. I would also like to thank Franco Rota for the many helpful conversations, and for patiently answering my questions. This paper was supported by the EPSRC (Engineering and Physical Sciences Research Council) via a postgraduate scholarship. 

\section{Type A}\label{section 2}
\renewcommand{\thetheorem}{\arabic{theorem}}
\setcounter{theorem}{0}
\numberwithin{theorem}{section}

Let $V$ be a complex vector space of dimension $2$. Choose a basis for $V$, and let $X,Y$ be the corresponding coordinate functions. We then identify $SL(V)$ with $SL_2(\CC)$. Take $\Gamma=C_n\subset SL_2(\CC)$ the cyclic group of order $n$, with $n\geq 2$, acting on $V$ via matrices of the form
\begin{equation*}
    \begin{pmatrix}
e^{2k\pi i/n } & 0\\
0 & e^{-2k\pi i/n }
\end{pmatrix}, \hspace{4mm} \text{with $1\leq k\leq n$}\, .    
\end{equation*}

The quotient $V/\Gamma$ is the Kleinian singularity of type $\mathbf{A_{n-1}}$. It is well-known that the algebra of functions $\CC[V/\Gamma]=\CC[V]^\Gamma$ is generated by the monomials $X^n,Y^n$ and $XY$, so we have
\begin{equation*}
    \CC[\mathbf{A_{n-1}}]=\CC[x,y,z]/(xy-z^n)\, ,
\end{equation*}
for all $n\geq 2$. Considering $X,Y$ to be of degree $1$, we have a grading on $\CC[\mathbf{A_{n-1}}]$ given by $\deg x=\deg y=n$ and $\deg z=2$. We also have a Poisson structure on $\CC[\mathbf{A_{n-1}}]$, induced by the symplectic structure on $V$; we choose here the normalization that gives the bracket $\{Y,X\}=1/n$ on $\CC[V]$. The induced Poisson structure on $\CC[\mathbf{A_{n-1}}]$ is then 

\begin{equation*}
    \{x,y\}=-nz^{n-1}\, , \hspace{5mm} \{z,x\}=x \hspace{5mm} \{z,y\}=-y\, .
\end{equation*}
This gives $\CC[\mathbf{A_{n-1}}]$ the structure of a graded Poisson algebra, with Poisson bracket of degree $-2$. 

\begin{oss}\label{oss generale poisson structure on C[x,y,z]}
    Let $\psi$ be a polynomial in $x,y,z$. It induces a Poisson structure on $\CC[x,y,z]$, defined by
    \begin{equation*}
        \{x,y\}=\pdv{\psi}{z}, \ \ \ \{x,z\}=-\pdv{\psi}{y}, \ \ \ \{y,z\}=\pdv{\psi}{x}.
    \end{equation*}
    Since the ideal $(\psi)\subset\CC[x,y,z]$ is Poisson, $\psi$ induces a Poisson structure on the quotient $\CC[x,y,z]/(\psi)$. If additionally $\psi$ is homogeneous with respect to some grading of $\CC[x,y,z]$, then $\CC[x,y,z]$ and $\CC[x,y,z]/(\psi)$ are graded Poisson algebras, and the Poisson bracket has degree $(\deg \psi-\deg x-\deg y-\deg z)$. The Poisson structure on $\CC[\mathbf{A_{n-1}}]$ is of this type, where $\psi=xy-z^n$.
\end{oss}

\subsection{Deformations and quantizations}

We have explicit presentations for both deformations and quantizations of the algebra $\CC[\mathbf{A_{n-1}}]$. In both cases, they are parameterised by a monic polynomial $P\in\CC[z]$ of degree $n$, with no degree $n-1$ term. 

Fix $n\geq2$ and denote by $\Aa(P)$ the deformation associated to the parameter $P$ of the algebra $\CC[\mathbf{A_{n-1}}]$. Explicitly, we have
\begin{equation*}
    \Aa(P)=\CC[x,y,z]/(xy-P(z)).
\end{equation*}
These algebras come from specializing the semi universal deformation of the Kleinian singularity (see for example \cite[Table 3]{katzGorensteinThreefoldSingularities1992}).
We give $\Aa(P)$ the structure of a Poisson algebra as in Remark \ref{oss generale poisson structure on C[x,y,z]}
\begin{equation*}
    \{x,y\}=\pdv{\psi}{z}=-P'(z), \ \ \ \{z,x\}=\pdv{\psi}{y}=x, \ \ \ \{z,y\}=-\pdv{\psi}{x}=-y.
\end{equation*}
The algebra $\Aa(P)$ is a filtered Poisson algebra, with filtration induced by the degree on the generators $\deg x=\deg y=n$ and $\deg z=2$, with Poisson bracket of degree $-2$, i.e. $\{F_k,F_m\}\subset F_{k+m-2}$ for all filtration terms $F_k,F_m$.

Quantizations of $\CC[\mathbf{A_{n-1}}]$ are examples of generalized Weyl algebras \cite{bavulaIsomorphismProblemsGroups2001}. We recall here the general definition.

\begin{defi}
Let $D$ be a ring, $\sigma$ an automorphism of $D$, and $a$ a central element of $D$. The generalized Weyl algebra $D(\sigma,a)$ is the ring extension of $D$ generated by two indeterminates $x,y$ subject to the relations
\begin{enumerate}
    \item $xd=\sigma(d)x$ and $yd=\sigma^{-1}(d)y$ for all $d\in D$,
    \item $xy=\sigma(a)$ and $yx=a$.
\end{enumerate}
\end{defi}

In our case, the quantization associated to the parameter $P$ is the generalized Weyl algebra 
\begin{equation*}
    A(P):=\CC[z](\sigma,P),
\end{equation*}
with $\sigma$ defined by $z\mapsto z-1$. Explicitly, $A(P)$ is the $\CC$-algebra generated by $x,y,z$ subject to the relations
\begin{equation}\label{eq defining A(P)}
    xz=(z-1)x, \ \ yz=(z+1)y, \ \ xy=P(z-1), \ \ yx=P(z). 
\end{equation}
The algebra $A(P)$ is a filtered associative algebra, with filtration induced by $\deg x=\deg y=n$ and $\deg z=2$. The commutator is of degree $-2$, i.e. $[F_k,F_m]\subset F_{k+m-2}$ for all filtration terms $F_k,F_m$.

We can give an explicit construction of the algebra $\Aa(P)$ as the semi-classical limit of $A(P)$. Define 
\begin{equation*}
    A_t(P):=\CC[t,z](\sigma_t,P),
\end{equation*}
with $\sigma_t$ defined by $z\mapsto z-t$, $ t\mapsto t$. Explicitly, $A_t(P)$ is the $\CC[t]$-algebra generated by $x,y,z$ subject to the following relations:
\begin{equation}\label{eq defining A_t(P)}
    xz=(z-t)x, \ \ yz=(z+t)y, \ \ xy=P(z-t), \ \ yx=P(z). 
\end{equation}
\begin{oss}
    The defining relations of $A(P)$ and $A_t(P)$ imply respectively that
    \begin{equation*}
        [x,y]=(\sigma-1)(P), \ \ \  \ \ \ [x,y]=(\sigma_t-1)(P).
    \end{equation*}
    Define for all $m\in\NN$ the operators $\delta_m:=(\sigma^m-1)$ and $\delta_{m,t}:=(\sigma_t^m-1)$, so that 
    \begin{equation*}
        [x,y]=\delta_1(P) \ \ \  \ \ \ [x,y]=\delta_{1,t}(P),
    \end{equation*}
    respectively in $A(P)$ and $A_t(P)$.
\end{oss}

Denote by $A_1(P):=A_t(P)/((t-1)A_t(P))$ and by $A_0(P):=A_t(P)/(tA_t(P))$.
\begin{prop}\label{prop semi-classical limit type A}
    The algebras $A_1(P) $ and $ A(P)$ are isomorphic as filtered associative algebras. The algebra $A_0$ has the structure of a filtered Poisson algebra, with Poisson bracket given by $\{\cdot,\cdot\}=\frac{1}{t}[\cdot, \cdot]$ mod $(t)$. Moreover, the algebras $A_0(P) $ and $ \Aa(P)$ are isomorphic as filtered Poisson algebras.
\end{prop}
\begin{proof}
    The isomorphism $A_1(P)\cong A(P)$ is trivial.
    
    Notice that $[A_t(P),A_t(P)]\subset tA_t(P)$ and $A_t(P)$ is free over $\CC[t]$, so $A_0(P)$ is commutative and $\{\cdot,\cdot\}=\frac{1}{t}[\cdot, \cdot]$ mod $(t)$ defines a Poisson bracket. We can see that it coincides with the one of $\Aa(P)$: the only non-trivial check is
    \begin{equation*}
        \frac{1}{t}[x,y] \! \! \mod (t)=\frac{P(z-t)-P(z)}{t}\!\!\mod (t)=-P'(z) \!\!\mod (t).
    \end{equation*}
\end{proof}

\begin{oss}
    If we instead take the associated graded of either $\Aa(P)$ or $A(P)$, we get back the algebra of functions on the Kleinian singularity $\CC[\mathbf{A_{n-1}}]$. 
\end{oss}

We introduce here the following notation, that will be used throughout the rest of the paper. If $\phi$ is an endomorphism of an algebra with three generators $x,y,z$, we will identify $\phi$ with the triple $(\phi(x),\phi(y),\phi(z))$ of its value on the generators $(x,y,z)$.

\begin{theorem}[{\cite[Theorem 3.9]{bavulaIsomorphismProblemsGroups2001}}]\label{teo automorfismi quantizzazione type A}
    The group of automorphisms of the algebra $A(P)$ has the following generators:
    \begin{enumerate}[(a)]
        \item $$\Phi_{\lambda,m}=\left(x+\sum_{i=1}^n{\frac{(-\lambda)^i}{i!}y^{im-1}\delta_m^i(P)}, y, z+m\lambda y^m \right),$$
        for every $\l\in\CC$ and $m\in\NN$;
        \item $$\Psi_{\lambda,m}=\left(x, y+\sum_{i=1}^n{\frac{\lambda^i}{i!}\delta_m^i(P)x^{im-1}}, z-m\lambda x^m\right),$$
        for every $\l\in\CC$ and $m\in\NN$;
        \item $$\Theta_\nu=(\nu x, \nu^{-1} y, z),$$
        for all $\nu\in\CC^\times$;
        \item $$\Omega=(y, (-1)^n x, 1-z), $$
        only if the polynomial $P$ is either odd or even, i.e. if $P(-z)=\pm P(z)$.
    \end{enumerate}
    
\end{theorem}

\begin{oss}
    A polynomial of degree $n$ is called reflective if there exists a $\rho\in\CC$ such that $P(\rho-x)=(-1)^nP(x)$. If $P$ is reflective, then $A(P)$ has an automorphism of the form $\Omega_\rho=(y, (-1)^n x, 1+\rho-z)$ \cite[Lemma 3.8]{bavulaIsomorphismProblemsGroups2001}. Notice though that if we restrict ourselves to polynomials with no term of degree $n-1$, $P$ can be reflective only if $\rho=0$. In that case, the polynomial $P$ is either odd or even and $\Omega_\rho=\Omega$. For brevity and to stick with the original notation of Bavula and Jordan, we will write ``$P$ is reflective'' instead of ``$P$ is either odd or even''. 
\end{oss}

\begin{oss}
    We can define the automorphisms of type $\Phi$ and $\Psi$ in a different way. For all $m\geq0$, consider $\ad(x^m)$ and $\ad(y^m)$, the adjoint actions of $x^m$ and $y^m$. These are nilpotent the derivations of $A(P)$ (see \cite[Lemma 3.4]{bavulaIsomorphismProblemsGroups2001}). If we exponentiate them, we get automorphisms of $A(P)$. In fact, 
    \begin{equation*}
        e^{\l\ad(x^m)}=\Phi_{\l,m} \text{ and } e^{\l\ad(y^m)}=\Psi_{\l,m} 
    \end{equation*}
    (see \cite[Lemma 3.4]{bavulaIsomorphismProblemsGroups2001}).
\end{oss}

We can lift the generators of $\Aut(A(P))$ to the algebra $A_t(P)$. Define the following $\CC[t]$-linear automorphisms of $A_t(P)$:
\begin{enumerate}[(a)]
    \item \begin{equation*}
        \Phi_{\l,m,t}:=e^{\l/t\ad(x^m)}=\left(x+\sum_{i=1}^n{\frac{(-\lambda)^i}{t^ii!}y^{im-1}\delta_{m,t}^i(P)}, y, z+m\lambda y^m \right),
    \end{equation*} 
    for all $\l\in\CC$ and $m\geq0$;
    \item \begin{equation*}
        \Psi_{\l,m,t}:=e^{\l/t\ad(y^m)}=\left(x, y+\sum_{i=1}^n{\frac{\lambda^i}{t^ii!}\delta_{m,t}^i(P)x^{im-1}}, z-m\lambda x^m\right),
    \end{equation*}
   for all $\l\in\CC$ and $m\geq0$;
    \item $\Theta_{\nu,t}=(\nu x, \nu^{-1} y, z)$ for all $\nu\in\CC^\times$;
    \item $\Omega_t=(y, (-1)^n x, t-z)$, only if $P$ is reflective.
\end{enumerate}

We can also consider how they act on the quotient $A_0$. Using that $\{\cdot,\cdot\}=\frac{1}{t}[\cdot, \cdot] \mod (t)$, we have
\begin{equation*}
    e^{\l\ad(x^m)}=\Phi_{\l,m,0} \text{ and } e^{\l\ad(y^m)}=\Psi_{\l,m.0},
\end{equation*}
where now ``$\ad$'' denotes the adjoint action with respect to the Poisson bracket.

Denote by $G_t$ the group generated by the automorphisms (a)-(d) of $A_t(P)$. Since the automorphisms in $G_t$ are $t$-linear, $G_t$ acts on the algebras $A_0(P)$ and $A_1(P)$. Denote the images of $G_t$ in these representations as $G_0$ and $G_1$ respectively. By Theorem \ref{teo automorfismi quantizzazione type A},
\begin{equation}\label{eq G_1=Aut}
    G_1=\Aut(A(P)).
\end{equation}
We also have
\begin{equation}\label{eq G_0 in PAut}
    G_0\leq \PAut(\Aa(P)).
\end{equation}
In fact, let $\phi\in G_0$ and $a,b\in\Aa(P)\cong A_0(P)$. Then
\begin{equation*}
    \phi(\{a,b\})=\frac{1}{t}\phi([\hat{a},\hat{b}]) \!\! \mod (t)=\frac{1}{t}[\phi(\hat{a}),\phi(\hat{b})] \!\! \mod (t)= \{\phi(a),\phi(b)\},
\end{equation*}
where $\hat{a},\hat{b}$ are some lifts of $a,b$ in $A_t(P)$.

\subsection{Amalgamated free product group structure}\label{section amalgamated free product}

The goal of this section is to prove Theorem \ref{teo amalgamated free product}, which gives us the explicit group structure of $G_t(P)$ as an abstract group.

The following proposition gives us some useful identities.

\begin{lemma}\label{lemma commuting relations G_t}
    The following relations hold in $G_t$.
    \begin{align}
        &\Theta_{\nu,t}\circ\Theta_{\mu,t}=\Theta_{\mu+\nu,t},  \label{eq lemma commuting relations G_t -1} \\
        &\Omega_t^2=\Theta_{(-1)^n,t},\label{eq lemma commuting relations G_t 0}\\
        &\Theta_{\nu,t}\circ\Psi_{\lambda,m,t}=\Psi_{\l\nu^m,m,t}\circ\Theta_{\nu,t},\label{eq lemma commuting relations G_t 1}\\
        &\Theta_{\nu,t}\circ\Phi_{\lambda,m,t}=\Phi_{\l\nu^{-m},m,t}\circ\Theta_{\nu,t},\label{eq lemma commuting relations G_t 2}
    \end{align}
    for all $\nu\in\CC^\times$, $\l\in\CC$ and $m\in\NN$.
    
    If $P$ is reflective, we also have
    \begin{align}
        &\Omega_t\circ \Theta_{\nu,t}=\Theta_{\nu^{-1},t}\circ\Omega_t \label{eq lemma commuting relations G_t 3} \\
        &\Theta_{(-1)^n,t}\circ \Omega_t\circ\Phi_{\lambda,m,t}\circ\Omega_t=\Psi_{\lambda,m,t}, \label{eq lemma commuting relations G_t 4}   
    \end{align}
    for all $\nu\in\CC^\times$, $\l\in\CC$ and $m\in\NN$.
    
    In particular, if $P$ is reflective, $G_t$ can be generated by just the automorphisms $\Theta_{\nu,t}$, $\Phi_{\l,m,t}$ and $\Omega_t$.
\end{lemma}
\begin{proof}
    Equations \eqref{eq lemma commuting relations G_t -1} and \eqref{eq lemma commuting relations G_t 0} follow from a direct check.
    
    To prove \eqref{eq lemma commuting relations G_t 1}, compute
    \begin{align*}
        \Theta_{\nu,t}\circ\Psi_{\lambda,m,t}&=\left(\nu x,\nu^{-1} y+\sum_{i=1}^n\frac{\l^i}{t^ii!}\nu^{im-1}\delta_{m,t}^i(P)x^{im-1}, z-\nu^m m\l x^m\right)\\
    &=\Psi_{\l\nu^m,m,t}\circ\Theta_{\nu,t}.
    \end{align*}
    Similarly, for equation \eqref{eq lemma commuting relations G_t 2} 
    \begin{align*}
        \Theta_{\nu,t}\circ\Phi_{\lambda,m,t}&=\left(\nu x+\sum_{i=1}^n\frac{(-\l)^i}{t^ii!}\nu^{1-im}x^{im-1}\delta_{m,t}^i(P),\nu^{-1} y, z+\nu^{-m} m\l y^m\right)\\
    &=\Phi_{\l\nu^{-m},m,t}\circ\Theta_{\nu,t}.
    \end{align*}
    
    Equation \eqref{eq lemma commuting relations G_t 3} follows from a direct check
    \begin{align*}
    \Omega_t\circ\Theta_{\nu,t}=(\nu y, (-1)^n\nu^{-1} x,t-z)=\Theta_{\nu^{-1},t}\circ\Omega_t.
\end{align*}
    Equation \eqref{eq lemma commuting relations G_t 4} is more involved. Let us check it separately on each of the three generators. For $x$ and $z$ we have
    $$x\mapsto y\mapsto y\mapsto (-1)^nx\mapsto x, $$
$$z\mapsto t-z\mapsto t-z-m\lambda y^m\mapsto z-(-1)^{n\cdot m}\l mx^m\mapsto z-m\l x^m. $$
Let us now check it for $y$:
\begin{equation}\label{eq lemma commuting relations G_t check 4 }
    \begin{aligned}
        y&\mapsto (-1)^nx \mapsto (-1)^n\left[x+\sum_{i=1}^n{\frac{(-\lambda)^i}{t^ii!}y^{im-1}\delta_{m,t}^i(P)}\right] \\
        &\mapsto (-1)^n\left[y+\sum_{i=1}^n{\frac{(-\lambda)^i}{t^ii!}(-1)^{(im-1)n}x^{im-1}(\delta_{m,t}^iP)(t-z)}\right] \\
        &\mapsto y+(-1)^n\sum_{i=1}^n{\frac{(-\lambda)^i}{t^ii!}x^{im-1}(\delta_{m,t}^iP)(t-z)}.
    \end{aligned}
\end{equation}
The final term in \eqref{eq lemma commuting relations G_t check 4 } is similar to $\Psi_{\l,m,t}(y)$, but we need to move $x^{im-1}$ to the right of $(\delta_{m,t}^iP)(t-z)$. By relations \eqref{eq defining A_t(P)}, we can do that by applying $\sigma^{im-1}_t$ to $(\delta_{m,t}^iP)(t-z)$. Notice that
\begin{equation*}
    (\delta_{m,t}^iP)(t-z)=\sigma_t\circ\gamma\circ(\sigma^m_t-1)^i(P),
\end{equation*}
where $\gamma$ is the map $z\mapsto -z$. Hence
\begin{equation*}
    (\delta_{m,t}^iP)(t-z)=(-1)^n\sigma_t\circ(\sigma^{-m}_t-1)^i(P),
\end{equation*}
because $P$ is reflective. It follows that 
\begin{equation*}
    \begin{aligned}
        \sigma^{im-1}_t((\delta_{m,t}^iP)(t-z))&=(-1)^n\sigma_t^{im}\circ(\sigma^{-m}_t-1)^i(P)\\
        &=(-1)^{n}(1-\sigma^m)^i(P)\\
        &=(-1)^{n+i}\delta^i_{m,t}(P).
    \end{aligned}
\end{equation*}
Putting this together with \eqref{eq lemma commuting relations G_t check 4 } we get

\begin{equation*}
    \Theta_{(-1)^n,t}\circ \Omega_t\circ\Phi_{\lambda,m,t}\circ\Omega_t(y)=y+\sum_{i=1}^n{\frac{\lambda^i}{t^ii!}\delta_{m,t}^i(P)x^{im-1}}=\Psi_{\l,m,t}(y).
\end{equation*}

This completes the proof.
\end{proof}

Let us recall the definition of the amalgamated free product. Let $G$ be a group and $H,K$ two subgroups of $G$, and let $L:=H\cap K$. The group $G$ is the free product of the subgroups $H$ and $K$ with the amalgamated subgroup $L$, and is denoted by $G=H\ast_{L}K$, if
    \begin{enumerate}[(a)]
        \item $G$ is generated by the subgroups $H$ and $K$;
        \item the defining relations of $G$ consist only of the defining relations of the subgroups $H$ and $K$.
    \end{enumerate}

\begin{theorem}\label{teo tecnico amalgamated free product}
Let $H,K$ be subgroups of $G$, and $L=H\cap K$. If $S_1$ is a set of left coset representatives for $L$ in $H$ and $S_2$ is a set of left coset representatives for $L$ in $K$, then $G=H\ast_{L} K$ if and only if every element $g\in G$ can be written uniquely as
$$g=g_1\dotso g_k\alpha, $$
where $\alpha\in L$, $g_i\in S_1\cup S_2$ and $g_i,g_{i+1}$ do not belong in $S_1$ and $S_2$ at the same time, for all $i=1,\dotso,k$.
\end{theorem}

A proof of Theorem \ref{teo tecnico amalgamated free product} can be found in \cite[Corollary 4.4.1]{solitarCombinatorialGroupTheory1976}. 

We want to show that $G_t, G_0$ and $G_1$ have an amalgamated free product structure. We introduce the following notation.
\begin{enumerate}[(i)]
    \item $\Phi:=\langle \Phi_{\l,m,t} \,| \,\l\in\CC, m\in\NN\rangle$
    \item $\Psi:=\langle \Psi_{\l,m,t} \,| \,\l\in\CC, m\in\NN\rangle$
    \item $\Theta:=\langle \Theta_{\nu,t} \,| \,\nu\in\CC^\times\rangle$
\end{enumerate}

Let us first consider the case when $P$ is reflective. Define 
\begin{enumerate}[(a)]
    \item $T:=\langle \Phi, \Theta \rangle$;
    \item $J:=\langle \Omega_t, \Theta \rangle $.
\end{enumerate}
Lemma \ref{lemma commuting relations G_t} implies that $G_t$ is generated by $T$ and $J$. Clearly, $T\cap J=\Theta$.

\begin{oss}
    We have $T=\Phi\rtimes\Theta $. This follows at once from relation \eqref{eq lemma commuting relations G_t 1}. Thus, the elements of $\Phi$ form a set of left coset representatives for $\Theta$ in $T$. Notice that, due to the properties of the exponential, each automorphism in $\Phi$ has the form $\exp^{1/t\ad(g(y))}$, for some $g\in\CC[y]$. For consistency with the notation used in Section \ref{section Poisson automorphisms type A}, we define
    \begin{equation}\label{def Phi_g}
        \Phi_{g,t}:=\exp^{1/t\ad(\hat{g}(y))},
    \end{equation}
    where $\hat{g}:=\int^y_0g(t)dt$ denotes the antiderivative of $g$. From the properties of the exponential it follows that
    \begin{equation*}
        \Phi_{g,t}\circ\Phi_{h,t}=\Phi_{g+h,t} \hspace{8mm} \Phi_{0,t}=\id,
    \end{equation*}
    so $\Phi\cong\CC[y]$ as an additive group via the identification $g\mapsto\Phi_{g,t}$.   
    
    From relation \eqref{eq lemma commuting relations G_t 3} we can also see that $\{\Omega_t,\id\}$ is a set of left coset representatives for $\Theta$ in $J$.
\end{oss}

\begin{prop}\label{prop decomposition G_t reflective}
If $P$ is reflective, all $\phi\in G_t$ can be written in the form
\begin{equation}\label{eq decomposition G_t reflective}
\phi=\Phi_{g_1,t}\circ\Omega_t\circ\dots\circ\Omega_t\circ\Phi_{g_s,t}\circ \Theta_{\nu,t},
\end{equation}
where $s\geq0$, $g_i\in\CC[y]$ for all $i$ and $0\neq g_i$ for $1<i<s$ , and $\nu\in\CC^\times$.
\end{prop}
\begin{proof}
By definition of $G_t$ we have 
\begin{equation}\label{eq proof decomposition G_t reflective}
    \phi=\phi_1\circ\dots\circ\phi_k,
\end{equation}
with $\phi_i$ either $\Phi_{g,t},\Omega_t$ or $\Theta_{\nu,t}$. From relations \eqref{eq lemma commuting relations G_t 1} and \eqref{eq lemma commuting relations G_t 3}, we can take every automorphism of type $\Theta_{\nu,t}$ to the right. Since $\Theta_{\nu,t}\circ \Theta_{\mu,t}=\Theta_{\nu\mu,t}$, $\Omega_t^2=\Theta_{(-1)^n,t}$ and $\Phi_g\circ\Phi_h=\Phi_{g+h}$, we can always rewrite \eqref{eq proof decomposition G_t reflective} as
$$\phi=\Phi_{g_1,t}\circ\Omega_t\circ\dots\circ\Omega_t\circ\Phi_{g_s,t}\circ \Theta_{\nu,t}, $$
with $s\geq0$, $g_i\in\CC[y]$ for all $i$ and $0\neq g_i$ for $1<i<s$ , and $\nu\in\CC^\times$.
\end{proof}

Let us now consider the case where $P$ is not reflective. Define 
\begin{enumerate}[(a)]
    \item $Q_1=\langle \Phi, \Theta\rangle$;
    \item $Q_2=\langle \Psi, \Theta\rangle$.
\end{enumerate}
By definition, $G_t$ is generated by $Q_1$ and $Q_2$. Clearly, $Q_1\cap Q_2=\Theta$. 

\begin{oss}
    We have $Q_1=\Phi\rtimes\Theta $ and $Q_2=\Phi\rtimes\Theta $. This follows at once from relation \eqref{eq lemma commuting relations G_t 2}. Thus, the elements of $\Phi$ (respectively $\Psi$) form a set of left coset representatives for $\Theta$ in $
    Q_1$ (respectively for $\Theta$ in $Q_2$). As in \eqref{def Phi_g} we can define
    \begin{equation*}
        \Psi_{g,t}:=\exp^{1/t\ad(\hat{g}(x))}.
    \end{equation*}
    This defines an isomorphism $\Psi\cong\CC[x]$ as additive groups via the map $g\mapsto \Psi_{g,t}$.
\end{oss}

\begin{prop}\label{prop decomposition G_t non reflective}
If $P$ is not reflective, every $\phi\in G_t$ can be written in the form
$$\phi=\Phi_{g_1,t}\circ\Psi_{h_1,t}\circ\dots\circ\Phi_{g_{s-1},t}\circ\Psi_{h_{s-1},t}\circ\Phi_{g_{s},t}\circ \Theta_{\nu,t}, $$
with $s\geq1$, $g_i\in\CC[y]$ for all $i$ with $0\neq g_i$ for $i=2,\dotso,s-1$, $0\neq h_i\in\CC[x]$ for all $i$, and $\nu\in\CC^\times$.
\end{prop}
\begin{proof}
    The proof is analogous to that of Proposition \ref{prop decomposition G_t reflective}, and it is a simple application of the relations \eqref{eq lemma commuting relations G_t -1}, \eqref{eq lemma commuting relations G_t 1} and \eqref{eq lemma commuting relations G_t 2}.
\end{proof}

We are left to prove that the decompositions of Propositions \ref{prop decomposition G_t reflective} and \ref{prop decomposition G_t non reflective} are unique. To do that, we introduce the notion of multidegree of an automorphism. For all $a\in A_t(P)$, we define $\deg(a)$ as the smallest natural number $i$ such that $a\in F_i$, with $\deg(0)=-\infty$. For every automorphism $\phi$ in $G_t$, we define its multidegree to be:
\begin{equation}
    \mdeg(\phi)=(\deg(\phi(x)),\deg(\phi(y)),\deg(\phi(z))).
\end{equation}

\begin{lemma}\label{lemma degree}
    For all $a,b\in A_t(P)$, $\deg(ab)=\deg(a)+\deg(b)$.
\end{lemma}
\begin{proof}
    For $a\in A_t(P)$, let $\gr(a):=a+F_{i-1}\in\gr(A_t(P))$, where $i=\deg(a)$. Clearly, $i=\deg(a)=\deg(\gr(a))$. Since $\gr(A_t(P))\cong\CC[t,x,y,z]/(xy-z^n)$ is a domain, $$\deg(\gr(a)\gr(b))=\deg(\gr(a))+\deg(\gr(b))$$ for all $a,b\in A_t(P)$. Thus
    \begin{equation*}
        \deg(ab)=\deg(\gr(ab))=\deg(\gr(a)\gr(b))=\deg(a)+\deg(b).
    \end{equation*}
\end{proof}

\begin{lemma}\label{lemma multidegree}
Let $g\in\CC[x]$ be a polynomial of degree $k$, and $\phi=(v_1,v_2,v_3)$ be an automorphism in $G_t$, such that $deg(v_2)>\deg(v_3)$ and $\deg(v_2)\geq\deg(v_1)$. Then $\phi\circ\Phi_{g,t}$ has multidegree
$$((nk+n-1)\deg(v_2), \deg(v_2), (k+1)\deg(v_2) ).  $$

Similarly, if $\deg(v_1)>\deg(v_3)$ and $\deg(v_1)\geq\deg(v_2)$, then $\phi\circ\Psi_{g,t}$ has multidegree 
$$(\deg(v_1), (nk+n-1)\deg(v_1), (k+1)\deg(v_1)). $$
\end{lemma}
\begin{proof}

Recall that $\Phi_{g,t}=e^{1/t\ad(\hat{g}(y))}$ \eqref{def Phi_g}. We have that 
\begin{equation}\label{eq adjoint action y^m}
    \ad(y^m): x\mapsto -y^{m-1}\delta_{m,t}(P), \ f(z)\mapsto -y^m\delta_{m,t}(f), \ y\mapsto 0
\end{equation}
for all $f\in\CC[z]$ (see \cite[Equation 8]{bavulaIsomorphismProblemsGroups2001}). This implies that $\Phi_{g,t}(y)=y$, hence $\deg(\phi\circ\Phi_{g,t}(y))=\deg(v_2)$.

Let $a y^{k+1}$ be the leading term of $\hat{g}$. From \eqref{eq adjoint action y^m}
\begin{equation*}
    \phi\circ\Phi_{g,t}(z)=v_3+(k+1)av_2^{k+1}+\text{lower terms in {$v_2$}. }
\end{equation*}
Since $\deg(v_2)\geq\deg(v_i)$ for $i=1,2,3$, from Lemma \ref{lemma degree}
\begin{equation*}
    \deg(\phi\circ\Phi_{g,t}(z))=(k+1)\deg(v_2).
\end{equation*}
 
From \eqref{eq adjoint action y^m} it follows that, for all $m,r\in\NN$
\begin{equation}\label{eq leading term 1}
   \relax [\hat{g}(y),y^mz^r]=\alpha t y^{k+m+1}z^{r-1}+\text{monomials proportional to $y^iz^j$,}
\end{equation}
with $i\leq k+m+1$, $j\leq r-1$ and $(i,j)\neq (k+m+1,r-1)$, for some $\alpha\in\CC^\times$. It also follows that
\begin{equation}\label{eq leading term 2}
    \relax [\hat{g}(y),x]=\beta t y^{k}z^{n-1}+\text{monomials proportional to $y^iz^j$,}
\end{equation}
with $i\leq k$, $j\leq n-1$ and $(i,j)\neq (k,n-1)$, for some $\beta\in\CC^\times$.

We now want to prove, by induction on $s$, that
\begin{equation}\label{eq lemma multidegree induction}
    \ad^s(\hat{g}(y))(x)=\alpha t^s y^{s(k+1)-1}z^{n-s}+\text{monomials proportional to $y^iz^j$,}
\end{equation}
with $i\leq s(k+1)-1$, $j\leq n-s$ and $(i,j)\neq (s(k+1)-1,n-s)$, for some $\alpha\in\CC^\times$. The base step is \eqref{eq leading term 1}. Consider now
\begin{equation*}
    \ad^{s+1}(\hat{g}(y))(x)=[\hat{g}(y),\alpha t^s y^{s(k+1)-1}z^{n-s}+\text{lower terms}],
\end{equation*}
for some $\alpha\in\CC^\times$, true by inductive hypothesis. Applying \eqref{eq leading term 1} we get \eqref{eq lemma multidegree induction}. It is then clear that
\begin{equation*}
    \deg{(\phi\circ\ad^s(\hat{g}(y))(x))}=\deg{(\alpha t^s v_2^{s(k+1)-1}v_3^{n-s})}.
\end{equation*}
Since $\deg(v_2)>\deg(v_3)$, we get that $\deg ({\phi\circ\Phi_{g,t}(x)})=\deg ({\alpha v_2^{nk+n-1}})$, hence from Lemma \ref{lemma degree}
\begin{equation*}
    \deg(\phi\circ\Phi_{g,t})=(nk+n-1)\deg(v_2).
\end{equation*}

The proof for $\phi\circ\Psi_{g,t}$ is analogous, using the relations 
\begin{equation}\label{eq adjoint action x^m}
    \ad(x^m): x\mapsto 0, \ f(z)\mapsto \delta_{m,t}(f)x^m, \ y\mapsto \delta_{m,t}(P)x^{m-1}
\end{equation}
\cite[Equation 6]{bavulaIsomorphismProblemsGroups2001} instead of \eqref{eq adjoint action y^m}.

\end{proof}

\begin{prop}\label{prop multidegree 2a}
Let $n>2$. An automorphism of the form 
$$\phi=\Phi_{g_1,t}\circ\Omega_t\circ\Phi_{g_2,t}\dots\circ\Omega_t\circ\Phi_{g_s,t},$$ with $s\geq1$, $0\neq g_i\in\CC[y]$ has multidegree
$$\left(n\prod_{i=1}^s(nk_i+n-1),n\prod_{i=1}^{s-1}(nk_i+n-1),n(k_s+1)\prod_{i=1}^{s-1}(nk_i+n-1)\right),$$
where $k_i=\deg(g_i)$.
\end{prop}
\begin{proof}
    We will prove this by induction on $s$. If $s=1$, then we have $\phi=\Phi_{g_1,t}$; since $\deg(y)=\deg(x)>\deg(z)$, we can apply Lemma \ref{lemma multidegree} on $\id\circ\Phi_{g_1,t}$. Thus
    \begin{equation*}
        \mdeg(\phi)=(n(nk+n-1),n,n(k+1)).
    \end{equation*}

When $s>1$, we have
$$\phi=\psi\circ\Omega_t\circ\Phi_{g_s,t},$$
with 
$$\psi=\Phi_{g_1,t}\circ\Omega_t\circ\Phi_{g_2,t}\dots\circ\Omega_t\circ\Phi_{g_{s-1},t}. $$

We know by the induction hypothesis that $\psi=(u_1,u_2,u_3)$ has multidegree
$$\left(n\prod_{i=1}^{s-1}(nk_i+n-1),n\prod_{i=1}^{s-2}(nk_i+n-1),n(k_{s-1}+1)\prod_{i=1}^{s-2}(nk_i+n-1)\right). $$
Notice that $\deg(u_2)\leq\deg(u_3)<\deg(u_1)$, because $(nk+n-1)>(k+1)\geq1$ for all $k\geq0$, when $n>2$. But now $\psi\circ\Omega_t=(u_2,(-1)^nu_1,t-u_3)$, hence we can once again apply lemma \eqref{lemma multidegree} to $(\psi\circ\Omega_t)\circ\Phi_{g_s,t}$ and complete the proof.
\end{proof}

\begin{prop}\label{prop multidegree 2b}
Let $n>2$. An automorphism of the form 
$$\phi=\Phi_{g_1,t}\circ\Psi_{h_1,t}\circ\dots\circ\Psi_{h_{s-1},t}\circ\Phi_{g_s,t},$$
with $s\geq1$, $0\neq g_i\in\CC[y]$ and $0\neq h_i\in\CC[x]$ has multidegree
\begin{equation*}
    \begin{pmatrix}
n(nk_s+n-1)\prod_{i=1}^{s-1}(nk_i+n-1)(nl_i+n-1) \\
 
\\
n\prod_{i=1}^{s-1}(nk_i+n-1)(nl_i+n-1) \\
\\
n(k_s+1)\prod_{i=1}^{s-1}(nk_i+n-1)(nl_i+n-1)
\end{pmatrix}
\end{equation*}
where $k_i=\deg(g_i)$ and $l_i=\deg(h_i)$.

\end{prop}
\begin{proof}
    We will prove this by induction on $s$. When $s=1$ we have $\phi=\Phi_{g_1,t}$, which has multidegree
    $(n(nk_1+n-1),n,n(k_1+1)). $
    
Let us now consider the case $s>1$. Take
$$\psi=\Phi_{g_1,t}\circ\Psi_{h_1,t}\circ\dots\circ\Phi_{g_{s-1},t}=(v_1,v_2,v_3), $$
so that $\phi=\psi\circ\Psi_{h_{s-1},t}\circ\Phi_{h_s,t}$.
By the induction hypothesis, we know that $\psi$ has multidegree 
\begin{equation*}
    \begin{pmatrix}
n(nk_{s-1}+n-1)\prod_{i=1}^{s-2}(nk_i+n-1)(nl_i+n-1) \\
 
\\
n\prod_{i=1}^{s-2}(nk_i+n-1)(nl_i+n-1) \\
\\
n(k_{s-1}+1)\prod_{i=1}^{s-2}(nk_i+n-1)(nl_i+n-1)
\end{pmatrix}
\end{equation*}

Since $n>2$, we have $nk_{s-1}+n-1>k_{s-1}+1\geq 1$, which implies that $\deg(v_2)\leq\deg(v_3)<\deg(v_1)$. We can apply again Lemma $\ref{lemma multidegree}$ to $\psi\circ\Psi_{h_{s-1},t}$ to get $\mdeg(\psi\circ\Psi_{h_{s-1},t})= $

\begin{equation*}
    \begin{pmatrix}
n(nk_{s-1}+n-1)\prod_{i=1}^{s-2}(nk_i+n-1)(nl_i+n-1) \\
 
\\
n\prod_{i=1}^{s-1}(nk_i+n-1)(nl_i+n-1) \\
\\
n(l_{s-1}+1)(nk_{s-1}+n-1)\prod_{i=1}^{s-2}(nk_i+n-1)(nl_i+n-1)
\end{pmatrix}
\end{equation*}
Let $\psi\circ\Psi_{h_{s-1},t}=(u_1,u_2,u_3)$. Then $\deg(u_2)>\deg(u_3)\geq\deg(u_1)$, and we can apply Lemma \ref{lemma multidegree} to $(\psi\circ\Psi_{h_{s-1},t})\circ\Phi_{g_s,t}$ to complete the proof.
\end{proof}

\begin{theorem}\label{teo amalgamated free product}
    Let $n>2$.
    \begin{enumerate}[(i)]
        \item If $P$ is reflective, the group $G_t$ is a free product with amalgamation
        \begin{equation*}
            G_t=T\ast_\Theta J.
        \end{equation*}
        As an abstract group, if $n$ is even
        \begin{equation*}
            G_t\cong (\CC[y]\rtimes\CC^\times)\ast_{\CC^\times} (\CC^\times\rtimes\ZZ/2\ZZ).
        \end{equation*}
        If $n$ is odd, then
        \begin{equation*}
            G_t\cong (\CC[y]\rtimes\CC^\times)\ast_{\CC^\times} H,
        \end{equation*}
        where $H=\langle \CC^\times, \Omega \, | \, \Omega^2=-1, \,\l\cdot\Omega_t=\Omega_t\cdot\l^{-1} \ \forall\l\in\CC^\times \rangle$.
        \item If $P$ is not reflective, the group $G_t$ is a free product with amalgamation
        \begin{equation*}
            G_t=Q_1\ast_\Theta Q_2.
        \end{equation*}
        As an abstract group,
        \begin{equation*}
            G_t\cong (\CC[y]\rtimes\CC^\times)\ast_{\CC^\times} (\CC[x]\rtimes\CC^\times).
        \end{equation*}
    \end{enumerate}
\end{theorem}
\begin{proof}
    Thanks to Theorem \ref{teo tecnico amalgamated free product}, we only need to prove that the decompositions in Propositions \ref{prop decomposition G_t reflective} and \ref{prop decomposition G_t non reflective} are unique.
    
    Consider first the case where $P$ is reflective. Assume, for contradiction, that
    \begin{equation*}
\phi=\Phi_{g_1,t}\circ\Omega_t\circ\dots\circ\Omega_t\circ\Phi_{g_s,t}=\Theta_{\nu,t},
    \end{equation*}
    for some $g_i\in\CC[y]$ such that $0\neq g_i$ for $1<i<s$, with $s\geq1$, and for some $\nu\in\CC^\times$. In particular, $\phi$ is a linear automorphism, since $\Theta_{\nu,t}$ is linear. Assume first that $g_1,g_s\neq0$. By Proposition \ref{prop multidegree 2a} we have that
    \begin{equation*}
        \deg(\phi(x))=n\prod_{i=1}^s(nk_i+n-1),
    \end{equation*}
    where $k_i=\deg(g_i)$.  So if $n>2$ then $ \deg(\phi(x))>n$, which leads to a contradiction. If $g_1$ is $0$, we can move $\Omega_t$ to the right to get
    \begin{equation*}
        \phi'=\Phi_{g_2,t}\circ\Omega_t\circ\dots\circ\Omega_t\circ\Phi_{g_s,t}=\Omega_t^{-1}\circ\Theta_{\nu,t}.
    \end{equation*}
    The right-hand side is still a linear automorphism, so we still get a contradiction. We can reason similarly if $g_s=0$.
    
    Let us assume now that $P$ is not reflective. Assume, by contradiction, that 
    \begin{equation}\label{eq decomposizione unica}
    \phi=\Phi_{g_1,t}\circ\Psi_{h_1,t}\circ\dots\circ\Phi_{g_{s-1},t}\circ\Psi_{h_{s-1},t}\circ\Phi_{g_{s},t} =\Theta_{\nu,t}
    \end{equation}
for some $g_i\in\CC[y]$ such that $0\neq g_i$ for $i=2,\dotso,s-1$, $0\neq h_i\in\CC[x]$, with $s\geq1$, and for some $\nu\in\CC^\times$. Assume first that $g_1,g_s\neq0$. Then by Proposition \ref{prop multidegree 2b} we have 
\begin{equation*}
    \deg(\phi(x))=n(nk_s+n-1)\prod_{i=1}^{s-1}(nk_i+n-1)(nl_i+n-1),
\end{equation*}
where $k_i=\deg(g_i)$ and $l_i=\deg(h_i)$. So if $n>2$ then $ \deg(\phi(x))>n$, which leads to a contradiction, because $\Theta_{\nu^{-1},t}$ is linear. 

If one or both of $g_1$ and $g_s$ are zero, then we can multiply both sides of \eqref{eq decomposizione unica} on the left and/or on the right by some $\Phi_{s_i,t}$, for some $0\neq s_i\in\CC[y]$. We get an equation of the form
\begin{equation*}
   \phi'=\Phi_{g_1,t}\circ\Psi_{h_1,t}\circ\dots\circ\Phi_{g_{s-1},t}\circ\Psi_{h_{s-1},t}\circ\Phi_{g_{s},t} =\Theta_{\nu,t}\circ\Phi_{q,t},
\end{equation*}
where $q$ is either $g_1, g_s$ or $g_1+g_s$, and $0\neq g_i$ for all $i$. We can now apply Proposition \ref{prop multidegree 2b} and get that $\deg(\phi' (y))=n\prod_{i=1}^{s-1}(nk_i+n-1)(nl_i+n-1)$, while $\deg(\Theta_{\nu^{-1},t}\circ\Phi_{q,t}(y))=n$. This leads to a contradiction, unless $s=1$. In that case though, equation \eqref{eq decomposizione unica} is either of the form $\id=\Theta_{\nu,t}$ or $\Phi_{g_1,t}=\Theta_{\nu^{-1},t}$, which again leads to contradiction. The decomposition of Proposition \ref{prop decomposition G_t non reflective} is thus unique.

The abstract group structures follow from the isomorphisms
\begin{equation*}
\begin{aligned}
        &\CC[y]\cong \Phi \hspace{15mm} &g(y)\mapsto\Phi_{g,t} \\
        &\CC[x]\cong \Psi \hspace{15mm} &g(x)\mapsto\Psi_{g,t} \\
    &\CC^\times\cong \Theta \hspace{15mm} &\nu\mapsto\Theta_{\nu,t},
\end{aligned}
\end{equation*}
and from the relations in Lemma \ref{lemma commuting relations G_t}.

\end{proof}

\begin{cor}
    Let $n>2$. The groups $G_0$ and $G_1$ have the same amalgamated free product structure described in Theorem \ref{teo amalgamated free product}. In particular, 
    \begin{equation*}
        G_0\cong G_1=\Aut(A(P)).
    \end{equation*}
\end{cor}
\begin{proof}
    Notice that the relations in Lemma \ref{lemma commuting relations G_t} still hold in $G_0$ and $G_1$, so we get decompositions analogous to those of Propositions \ref{prop decomposition G_t reflective} and \ref{prop decomposition G_t non reflective}. We can define a notion of degree for basis monomials in $A(P)$ and $\Aa(P)$ just as for $A_t(P)$. Notice also that in the proof of Lemma \ref{lemma multidegree}, no $t$ appears in the coefficients of the leading words, so the result holds in $G_0$ and $G_1$ too. These two lemmas were the only ingredient used in the proofs of Propositions \ref{prop multidegree 2a} and \ref{prop multidegree 2b}, and subsequently of Theorem \ref{teo amalgamated free product}. So, $G_0$ and $G_1$ have the same amalgamated free product structure as of $G_t$. In particular, they are isomorphic. 
\end{proof}

\subsection{The group of Poisson automorphisms of $\Aa(P)$}\label{section Poisson automorphisms type A}

We know that $G_0\leq \PAut(\Aa(P))$ by \eqref{eq G_0 in PAut}. In this section we show that this is an equality.

The group of automorphisms of $\Aa(P)$ as an affine variety is well known. Makar-Limanov computed its generators \cite{makar-limanovGroupsAutomorphismsClass1990}, and Blanc and Dubouloz \cite{blancAutomorphismsPreFibered2011} proved that it has an amalgamated free product structure.

\begin{theorem}[{\cite[Theorem 1]{makar-limanovGroupsAutomorphismsClass1990}}]\label{teo automorphism affine type A}
Let $R$ be the quotient algebra 
\begin{equation*}
    R=\CC[x,y,z]/(xy-P(z)),
\end{equation*}
with $P(z)\in\CC[z]$. Then the group $Aut(R)$ is generated by the following automorphisms:
\begin{enumerate}[(a)]
    \item Hyperbolic rotations: $\Theta_{\nu,0}=(\nu x,\nu^{-1} y, z)$, for all $\nu\in\CC^\times$;
    \item Involution: $V=(y,x,z)$;
    \item Triangular automorphisms: $$\Delta_g=(z+[P(z+yg(y))-P(z)]y^{-1}, y, z+yg(y))$$
    for all $g(y)\in\CC[y]$;
    \item \emph{(If $P(z)=c(z+a)^n$)} Rescalings $R_\nu=(\nu^n x, y,\nu z+(\nu-1)a)$, for all $\nu\in\CC^\times$;
    \item \emph{(If $P(z)=(z+a)^iQ((z+a)^d)$)} Symmetries: $S_\mu=(\mu^ix, y, \mu z+(\mu-1)a)$, for all $\mu^d=1$.
\end{enumerate}
\end{theorem}

\begin{lemma}\label{lemma automorphism affine type A}
    If $P$ is a monic polynomial of degree $n$ with no term of degree $n-1$, we can substitute (d) and (e) from the list of generators in Theorem \ref{teo automorphism affine type A} with
    \begin{enumerate}[(a)]\setcounter{enumi}{3}
        \item \emph{(If $P(z)=z^n$)} Rescalings $R_\nu=(\nu^n x, y,\nu z)$, for all $\nu\in\CC^\times$;
        \item \emph{(If $P(z)=z^iQ(z^d)$)} Symmetries: $S_\mu=(\mu^ix, y, \mu z)$, for all $\mu^d=1$.
    \end{enumerate}
\end{lemma}
\begin{proof}
    For (d), $P(z)=c(z+a)^n$ only if $c=1$ and $a=0$, since $P$ is monic and has no term of degree $n-1$.
    
    For (e), first notice that the polynomial $Q$ must be monic, since $P$ is. For degree reason, $n=i+dk$, where $k$ is the degree of $Q$. We can ignore the case $d=1$, since the only corresponding automorphism is $S_1=\id$. Expanding $P(z)$ we get $$(z+a)^{i+dk}+\alpha(z+a)^{i+d(k-1)}+\text{terms of lower degree}$$
    for some $\alpha\in\CC$. Since $d>1$, the term of degree $n-1=i+dk-1$ of $P(z)$ comes only from the expansion of $(z+a)^{i+dk}$. Since $P(z)$ has no term of degree $n-1$, we must have $a=0$.
\end{proof}

\begin{oss}\label{oss S special case of R}
    Notice that if $P(z)=z^iQ(z^d)$ is of degree $n$, and $\mu^d=1$, then $\mu^i=\mu^n$. So, if they are both defined, the automorphisms of type (e) are actually a special type of automorphisms of type (d), i.e. $R_\mu=S_\mu$ for all $\mu$ with $\mu^d=1$.
\end{oss}

Let us introduce the affine automorphism 
\begin{align*}
    &\nabla_g=V\circ\Delta_{-g}\circ V\\
    &=(x, y+[P(z-xg(x))-P(z)]x^{-1}, z-xg(x)).
\end{align*}

\begin{prop}
Let $g\in\CC[x]$. Then
$$\Delta_g=e^{\ad(\hat{g}(y))}=\Phi_{g,0},$$
$$\nabla_g=e^{\ad(\hat{g}(x))}=\Psi_{g,0}, $$
where $\hat{g}:=\int^x_0gdx$ denotes the antiderivative of $g$.
In particular, $\Delta_g$ and $\nabla_g$ are Poisson automorphisms and elements of $G_0$.
\end{prop}
\begin{proof}
We have that
\begin{equation}\label{eq poisson adjoint action y^n}
    \{y^k, f(z)\}=kf'(z)y^{k},\hspace{8pt} \{y^k,x\}=kP'(z)y^{k-1},
\end{equation} and 
\begin{equation}\label{eq poisson adjoint action x^n}
    \{x^k, f(z)\}=kx^kf'(z), \hspace{8pt}\{x^k,y\}=-kx^{k-1}P'(x),
\end{equation}
for all $k\geq0$. These relations can be checked directly or derived from \eqref{eq adjoint action y^m} and \eqref{eq adjoint action x^m} using the isomorphism $\Aa(P)\cong A_0(P)$. By \eqref{eq poisson adjoint action y^n}:
\begin{align*}
    &\{\hat{g}(y),z\}=yg(y), \\
    &\{\hat{g}(y),x\}=P'(z)g(y),\\
    &\{\hat{g}(y), P^{(i)}(z)y^{i-1}g^i(y)\}=P^{(i+1)}(z)y^ig^{i+1}(y),
\end{align*}
where $P^{(i)}$ denotes the $i$-th derivative of $P$. Hence, 
$$e^{\ad(\hat{g}(y))}=(x+\sum_{i=1}^n{\frac{1}{i!}P^{(i)}(z)y^{i-1}g^i(y)},y, z+yg(y)). $$
To complete the proof, we need to show that $\Phi_{g,0}$ acts on $x$ in the same way as $\Delta_g$. Let $P(z)=\sum_{i=0}^n{a_iz^i}$, with $a_n=1$ and $a_{n-1}=0$.
\begin{align*}
    &P(z+yg(y))-P(z)=\sum_{i=0}^n{a_i(z+yg(y))^i}-\sum_{i=0}^n{a_iz^i}= \\
    &\sum_{i=1}^n{\binom{n}{i}z^{n-i}y^ig^i(y)}+a_{n-2}\sum_{i=1}^{n-2}{\binom{n}{i}z^{n-2-i}y^ig^i(y)}+\dotso+a_1yg(y)= \\
    &\sum_{i=1}^n{y^ig^i(y)\left[\binom{n}{i}z^{n-i}+a_{n-2}\binom{n-2}{i}z^{n-2-i}+\dotso+a_i\right]}=\\
    &\sum_{i=1}^n{\frac{1}{i!}P^{(i)}(z)y^ig^i(y)}.
\end{align*}
A similar computation using \eqref{eq poisson adjoint action x^n} shows that $$\nabla_g=e^{\ad(\hat{g}(x))}=(x, y+\sum_{j=1}^n{\frac{(-1)^j}{j!}P^{(j)}(z)x^{j-1}g^j(A)}, z-xg(x)).$$
\end{proof}

\begin{oss}
    Consider the automorphism $S_{-1}=((-1)^nx,y,-z)$. This is in $\Aut(\Aa)$ if and only if $P(z)$ is of the form $z^iQ(z^2)$, by Proposition \ref{lemma automorphism affine type A}, i.e. if and only if $P(z)$ is reflective. If $P$ is reflective, then
    $$\Omega_0=S_{-1}\circ V=(y,(-1)^n x,-z). $$
\end{oss}

\begin{theorem}\label{teo generatori poisson type a}
The group $\PAut(\Aa(P))$ of Poisson automorphisms of $\Aa(P)$ is equal to $G_0$.
\end{theorem}
\begin{proof}
It is sufficient to prove that $\PAut(\Aa(P))$ is generated by $\Theta_{\nu,0}$, $\Delta_g$ and $\Omega_0$ when $P$ is reflective, and by $\Theta_{\nu,0}$, $\Delta_g$ and $\nabla_h$ when $P$ is not reflective, for all $\nu\in\CC^\times$ and $g,h\in\CC[x]$.

Let us compute some relations between the generators in $\Aut(\Aa(P))$. The following relations hold whenever the automorphisms are defined.
\begin{enumerate}
    \item $V^2=\id$
    \item $\Theta_{\nu,0}\circ V=V\circ \Theta_{\nu^{-1},0}$
    \item $\Delta_g\circ V=V\circ\nabla_{-g}$ and $\nabla_g\circ V=V\circ\Delta_{-g}$
    \item $R_{\l}\circ R_{\mu}=R_{\l\cdot\mu}$ and $S_\l\circ S_\mu=S_{\l\cdot\mu}$
    \item $\Theta_{\nu,0}\circ R_\mu=R_\mu\circ \Theta_{\nu,0}$ and $ \Theta_{\nu,0}\circ S_\mu=S_\mu\circ \Theta_{\nu,0}$
    \item $V\circ R_\nu=R_\nu\circ V\circ \Theta_{\nu^n,0}$ and $V\circ S_\mu=S_{\mu}\circ V\circ \Theta_{\mu^n,0}$.
\end{enumerate}

Suppose now that $P$ is of the form $P(z)=z^iQ(z^d)$. We have:

\begin{equation}
    \Delta_g\circ S_\mu
    =\left(\mu^ix+\mu^i\sum_{j=1}^n{\frac{1}{j!}P^{(j)}(z)C^{j-1}g^j(y)},y, \mu z+\mu yg(y)\right)
\end{equation}
\begin{equation}\label{eq composition aut 1}
    S_\mu\circ \Delta_{\mu g}=\left(\mu^i x+\sum_{j=1}^n{\frac{\mu^j}{j!}P^{(j)}(\mu z)y^{j-1}g^j(y)},y, \mu z+\mu yg(y)\right).
\end{equation}
Since $P(z)=z^iQ(z^d)$, the $j$-th derivative $P^{(j)}$ has only terms of degree equal to $d(k-l)+i-j$, where $k$ is the degree of $Q$, for some $l\geq0$. Hence each term of the sum in \eqref{eq composition aut 1} has a factor of $\mu^{d(k-l)+i}=\mu^i$. Thus 
\begin{enumerate}\setcounter{enumi}{6}
    \item $\Delta_g\circ S_\mu=S_\mu\circ\Delta_{\mu g}$.
\end{enumerate}

Similarly, we get
\begin{align*}
    & \nabla_g\circ S_\mu=(\mu^i x, y+\sum_{j=1}^n{\frac{(-1)^j}{j!}P^{(j)}(z)x^{j-1}g^j(x)}, \mu z-\mu xg(x))\\
    &S_\mu\circ \nabla_{\mu^{1-i} g(\mu^{-i}x)}=(\mu^i x, y+\sum_{j=1}^n{\frac{(-1)^j}{j!}P^{(j)}(\mu z)\mu^{j-i}x^{j-1}g^j(x)},\mu z-\mu xg(x)).
\end{align*}
Again, from $P^{(j)}(\mu z)$ we can take out $\mu$, getting factors of the form $\mu^{d(k-l)+i-j}$, which is equal to $\mu^{i-j}$. Thus
\begin{enumerate}\setcounter{enumi}{7}
    \item $\nabla_g\circ S_\mu=S_\mu\circ \nabla_{\mu^{1-i} g(\mu^{-i}x)}.$
\end{enumerate}

Suppose now that $P(z)=z^n$. We have:
\begin{enumerate}\setcounter{enumi}{8}
    \item $\Delta_g\circ R_\nu=R_\nu\circ \Delta_{\nu g}$, since they are both equal to 
    \begin{equation*}
        (\nu^n x+\nu^n\sum_{i=1}^n{\binom{n}{i}z^{n-i}y^{i-1}g^i(y)},y, \nu z+\nu yg(y))
    \end{equation*}
    \item $\nabla_g\circ R_\nu=R_\nu\circ\nabla_{\nu^{1-n} g(\nu^{-n}x)}$, since they are both equal to 
    \begin{equation*}
        (\nu^n x, y+\sum_{j=1}^n{(-1)^j\binom{n}{j}z^{n-j}x^{j-1}g^j(x)}, \nu z-\nu xg(x) ).
    \end{equation*}
\end{enumerate}

Take any automorphism $\psi\in\Aut(\Aa(P))$. From Theorem \ref{teo automorphism affine type A} and Proposition \ref{lemma automorphism affine type A} we know that we can write it as a composition of automorphisms of the form $V,\Theta_{\nu,0},\Delta_g,S_\mu $ or $R_\mu$. Using relations (1) to (10) we can rewrite $\psi$ as $\psi=\omega\circ\phi$, where $\phi$ is in $G_0$ and $\omega$ is one of the following automorphisms: $\id, V, R_\nu, S_\mu, R_\nu\circ V$ or $S_\mu\circ V$. If $\omega=\id$ then $\phi\in G_0$. If $\nu$ or $\mu$ are equal to $(-1)$, then $R_{-1}\circ V=S_{-1}\circ V=\Omega_0$. Since $\Omega_0\in G_0$ then $\psi\in G_0$. We want to show that in all the other cases, $\omega$ is not a Poisson automorphism. It follows from a direct check:
$$\{R_\nu(z),R_\nu(y)\}=-\nu y\neq -y= R_\nu (\{z,y\}), $$
for all $\nu\neq 1$,
$$\{ S_\mu(z),S_\mu(y)\} =-\mu y\neq -y=S_\mu(\{ z,y\} ), $$
for all $\mu\neq 1$,
$$\{S_\mu\circ V(z),S_\mu\circ V(x)\}=-\mu y\neq y=R_\mu\circ V(\{z,x\}), $$
for $\mu\neq-1$, 
$$\{R_\nu\circ V(z),R_\nu\circ V(x)\}=-\nu y\neq y=R_\nu\circ V(\{z,x\}), $$
for $\nu\neq-1$.

Hence, for all $\psi\in\Aut(\Aa(P))$, either $\psi\in G_0$ or $\psi$ is a composition of a Poisson automorphism and a non-Poisson automorphism, hence $\psi$ is non-Poisson.

\end{proof}

From the relations in the proof of Theorem \ref{teo generatori poisson type a} we also get:

\begin{cor}
The subgroup $G$ of Poisson automorphisms is normal in $Aut(\Aa)$.
\end{cor}

\begin{theorem}\label{teo poisson automorphisms type A}
    Let $n>2$. Then $\Aut(A(P))$ and $\PAut(\Aa(P))$ admit the amalgamated free product structure described in Theorem \ref{teo amalgamated free product} and
    \begin{equation*}
        \Aut(A(P))\cong\PAut(\Aa(P)),
    \end{equation*}
    for every deformation parameter $P$.
\end{theorem}
\begin{proof}
    It follows directly from Theorem \ref{teo amalgamated free product} and Theorem \ref{teo generatori poisson type a}.
\end{proof}

The case $n=2$ is summarised in the following theorem, due to the work of Dixmier \cite{dixmierQuotientsSimplesAlgebre1973}, O'Fleury \cite{fleurySousgroupesFinisAut1998} and Naurazbekova and Umirbaev \cite{naurazbekovaAutomorphismsSimpleQuotients2021}.

\begin{theorem}
    For $n=2$ and for every deformation parameter $P$, we have a group isomorphism
    \begin{equation*}
        \Aut(A(P))\cong\PAut(\Aa(P)),
    \end{equation*}
    and they have an amalgamated free product structure given by
    \begin{equation*}
        L\ast_{L\cap T}T,
    \end{equation*}
    where $L$ is the subgroup of linear (Poisson) automorphisms and $T$ is the subgroup generated by triangular (Poisson) automorphisms and hyperbolic rotations. 
    
\end{theorem}

\begin{oss}
    The main difference for $n=2$ is that triangular automorphisms $\Delta_g$ can be linear (it is if and only if $g$ is a constant). This is a consequence of the fact that the case $n=2$ is the only one where the Poisson algebra structure is linear on the generators $x,y,z$, inducing a Lie algebra isomorphic to $\ss\ll_2$.
\end{oss}

\subsection{Isomorphism groupoids (type $\mathbf{A}$)}

In this section we compute $\Iso(A)$ and $\PIso(\Aa)$, the groupoid of (Poisson) isomorphisms between quantizations and deformations of $\CC[\mathbf{A_{n-1}}]$, respectively, and verify Conjecture \ref{conj main} for type $\mathbf{A}$.

Let us first look at the quantization case.
\begin{theorem}[{\cite[Theorem 3.28]{bavulaIsomorphismProblemsGroups2001}}]\label{teo iso generalized weyl algebras}
    Let $P_1,P_2\in\CC[z]$. Then the generalized Weyl algebras $A(P_1),A(P_2)$ are isomorphic if and only if there exist some $\eta,\a\in\CC$, with $\eta\neq0$, such that $P_2(z)=\eta P_1(\pm z+\a)$.
\end{theorem}

\begin{cor}\label{cor iso quantizzazion type A}
    For all $n\geq2$, the groupoid $\Iso(A)$ is generated by the groups $\Aut(A(P))$, for all deformation parameters $P$, and by the isomorphism 
    \begin{equation*}
        \Omega:A(P(z))\rightarrow A((-1)^nP(-z)), \quad \Omega:=(y,(-1)^n x,1-z).
    \end{equation*}
\end{cor}
\begin{proof}
    Let $P_1$ be a deformation parameter, and $\eta,\a\in\CC$, with $\eta\neq0$. Suppose that $P_2=\eta P_1(\pm z+\a)$ is a deformation parameter. Then $\a=0$ because $P_1$ and $P_2$ have no term of degree $n-1$, and $\eta=(\pm 1)^n$ because $P_2$ and $P_1$ are monic. The claim now follows from Theorem \ref{teo iso generalized weyl algebras}.
\end{proof}

Let us consider the deformation case. We will first compute a set of generators for $\Iso(\Aa)$, the groupoid of affine isomorphisms between deformations. 

In \cite{blancAutomorphismsPreFibered2011} the authors study the isomorphisms between $\AA^1$-fibered surfaces. An affine surface $X$ is $\AA^1$-fibered if there exists a surjective morphism of affine varieties $\pi_X:X\rightarrow \AA^1$, with general fibres isomorphic to $\AA^1$, where $\AA^1$ denotes the complex affine line. Two $\AA^1$-fibered surfaces $(X,\pi)$ and $(X',\pi')$ are isomorphic as $\AA^1$-fibered surfaces if there are affine isomorphisms $f:X\rightarrow X'$ and $g:\AA^1\rightarrow\AA^1$ such that
\begin{equation*}
    \pi'\circ f=g\circ \pi.
\end{equation*}
Let $P\in\CC[z]$ be a polynomial of degree $n$. The algebra $\Aa(P)$ is $\AA^1$-fibered by $\p_P:\Aa(P)\rightarrow \CC[x]$, the projection on the variable $x$. 

\begin{lemma}\label{lemma iso fibered surfaces}
    Let $P,Q\in\CC[z]$ be of degree $n$. 
    \begin{enumerate}[(i)]
        \item Every fibration $\p ':\Aa(P)\rightarrow\AA^1$  is of the form $\p_P\circ\ga$, where $\ga$ is an affine automorphism of $\Aa(P)$;
        \item $(\Aa(P),\p_P)\cong(\Aa(Q),\p_Q)$ if and only if $Q(z)=\eta P(\a z+\beta)$, for some $\eta,\a,\b\in\CC$, with $\eta\neq0$.
    \end{enumerate}
\end{lemma}
\begin{proof}
    A proof can be found in \cite[Theorem 5.4.5]{blancAutomorphismsPreFibered2011}. Point $(i)$ was first proved by Daigle in \cite{daigleLocallyNilpotentDerivations2003}.
\end{proof}

\begin{prop}\label{prop affine iso deformation type A}
    For all $n\geq2$, the groupoid $Iso(A)$ is generated by the groups $\Aut(\Aa(P))$, for all deformation parameters $P$, and by the isomorphisms
    \begin{equation*}
        R_\a:\Aa(P(z))\rightarrow \Aa(\a^{-n}P(\a z)),  \quad R_\a:=(\a^{n} x,y,\a z), \quad \forall\a\in\CC^\times.
    \end{equation*}
\end{prop}
\begin{proof}
    Let $P,Q$ be two deformation parameters, and $\phi:\Aa(P)\rightarrow\Aa(Q)$ be an affine isomorphism. Then $\phi$ is a fibered isomorphism between
    \begin{equation*}
        \phi:(\Aa(P),\pi_Q\circ\phi)\rightarrow(\Aa(Q),\pi_Q).
    \end{equation*}
    From Lemma \ref{lemma iso fibered surfaces} there exists $\ga\in\Aut(\Aa(P))$ such that $\pi_Q\circ\phi=\pi_P\circ\ga$. Thus $\psi:=\phi\circ\ga^{-1}:(\Aa(P),\pi_P)\rightarrow(\Aa(Q),\pi_Q)$ is a fibered isomorphism. By Lemma \ref{lemma iso fibered surfaces}, there exist $\eta,\a,\b\in\CC$, with $\eta\neq0$, such that $Q(z)=\eta P(\a z+\b)$. Since $P$ and $Q$ are deformation parameters, $\b=0$ and $\a^n=\eta$. 

    Consider the fibered isomorphism
    \begin{equation*}
        R_\a:(\Aa(P),\p_P)\rightarrow (\Aa(Q),\p_Q)  \quad R_\a:=(\a^{n} x,y,\a z).
    \end{equation*}
    Clearly $R_\a^{-1}\circ \psi\in\Aut(\Aa(P))$, so $\phi$ can be written as composition of $R_\a$ and an automorphism of $\Aa(P)$.
\end{proof}

\begin{theorem}\label{teo Poisson iso type A}
    For all $n\geq2$, the groupoid $\PIso(\Aa)$ is generated by the groups $\PAut(\Aa(P))$, for all deformation parameters $P$, and by the isomorphism 
    \begin{equation*}
        \t:\Aa(P(z))\rightarrow \Aa((-1)^nP(-z)) \quad \t:=(y,(-1)^n x,-z).
    \end{equation*}
\end{theorem}
\begin{proof}
    Let $P,Q$ be two deformation parameters, and $\phi:\Aa(P)\rightarrow\Aa(Q)$ be a Poisson isomorphism. In particular, $\phi$ is an affine isomorphism, so, by Proposition \ref{prop affine iso deformation type A}, $\phi=R_\a\circ\ga$, for some $\ga\in\Aut(\Aa(P))$. By the proof of Theorem \ref{teo generatori poisson type a}, $\ga$ can be written as $\omega\circ\xi$, where $\xi\in\PAut(\Aa(P))$ and $\omega$ is one of the following affine automorphisms: $\id, V, R_\nu, S_\mu, R_\nu\circ V$ or $S_\mu\circ V$. Since $\phi$ is Poisson, $R_\a\circ\omega$ needs to be too. With the same computations as in the proof of Theorem \ref{teo generatori poisson type a}, we can check that the only case where $R_\a\circ\omega$ is Poisson is $\a=-1$ and $\omega=V$. In that case, $R_{-1}\circ V=\tau$.
\end{proof}

Putting together Corollary \ref{cor iso quantizzazion type A} and Theorem \ref{teo Poisson iso type A} we have verified Conjecture \ref{conj main} for type $\mathbf{A}$.

\begin{theorem}\label{teo iso type A}
    Let $n\geq 4$. We have an isomorphism of groupoids
    \begin{equation*}
        \Iso(A)\cong\PIso(\Aa).
    \end{equation*}
\end{theorem}

\section{Type D}\label{section 3}

Once again, let $V$ be a complex vector space of dimension $2$. Choose a basis for $V$, and let $X,Y$ be the corresponding coordinate functions. We then identify $SL(V)$ with $SL_2(\CC)$. Take $\Gamma=BD_{n-2}\subset SL_2(\CC)$ the binary dihedral group of order $4(n-2)$, with $n\geq 4$. The group $\Gamma$ is generated by
\begin{equation*}
    \sigma:=\begin{pmatrix}
e^{\pi i/(n-2) } & 0\\
0 & e^{-\pi i/(n-2) }
\end{pmatrix}, \hspace{4mm}    \tau:=\begin{pmatrix}
0 & 1\\
-1 & 0
\end{pmatrix}.
\end{equation*}

The quotient $V/\Gamma$ is the Kleinian singularity of type $\mathbf{D_{n}}$. The algebra of functions $\CC[V/\Gamma]=\CC[V]^\Gamma$ is generated by the polynomials $X^2Y^2,(X^{2(n-2)}+Y^{2(n-2)})$ and $XY(X^{2(n-2)}-Y^{2(n-2)})$. We have
\begin{equation*}
    \CC[\mathbf{D_{n}}]\cong\CC[x,y,z]/(x^{n-1}+xy^2+z^2)\, ,
\end{equation*}
for all $n\geq 4$. Considering $X,Y$ to be of degree $1$, we have a grading on $\CC[\mathbf{D_{n}}]$ given by $\deg x=4$, $ \deg y=2(n-2)$ and $\deg z=2(n-1)$. We give $\CC[\mathbf{D_{n}}]$ the structure of a graded Poisson algebra as described in Remark \ref{oss generale poisson structure on C[x,y,z]}, with Poisson bracket of degree $-2$. This structure is equivalent to the one induced by the standard symplectic structure on $V$, after a suitable renormalization. Explicitly, we have

\begin{equation*}
    \{x,y\}=2z\, , \hspace{5mm} \{x,z\}=-2xy, \hspace{5mm} \{y,z\}=(n-1)x^{n-2}+y^2\, .
\end{equation*}

\subsection{Deformations and quantizations}

We have explicit presentations for both deformations and quantizations of the algebra $\CC[\mathbf{D_{n}}]$. In both cases, they are parameterised by a pair $(Q,\gamma)$, where $Q\in\CC[x]$ is a monic polynomial of degree $n-1$, and $\gamma\in\CC$.

Fix $n\geq4$ and denote by $\Dd_n(Q,\gamma)$ the deformation associated to the parameter $(Q,\gamma)$ of the algebra $\CC[\mathbf{D_{n}}]$. Explicitly, we have
\begin{equation*}
    \Dd_n(Q,\gamma)=\CC[x,y,z]/(Q(x)+xy^2+z^2-\gamma y).
\end{equation*}
These algebras come from specializing the semi-universal deformation of the Kleinian singularity (see for example \cite[Table 3]{katzGorensteinThreefoldSingularities1992}).
We give $\Dd_n(Q,\gamma)$ the structure of a Poisson algebra as in Remark \ref{oss generale poisson structure on C[x,y,z]},
\begin{equation*}
    \{x,y\}=\pdv{\psi}{z}=2z, \ \ \ \{x,z\}=-\pdv{\psi}{y}=-2xy+\gamma, \ \ \ \{y,z\}=\pdv{\psi}{x}=Q' (x)+y^2.
\end{equation*}
The algebra $\Dd_n(Q,\gamma)$ is a filtered Poisson algebra, with filtration induced by $\deg x=4 $, $\deg y=2(n-2)$ and $\deg z=2(n-1)$, with Poisson bracket of degree $-2$, i.e. $\{F_k,F_m\}\subset F_{m+k-2}$ for all filtration terms $F_k,F_m$.

We denote by $D_n(Q,\ga)$ the quantization associated to the parameter $(Q,\ga)$. The algebras $D_n(Q,\ga)$ were studied by Levy in \cite{levyIsomorphismProblemsNoncommutative2009}. The algebra $D_n(Q,\ga)$ is the $\CC$-algebra generated by $x,y,z$, subject to the relations
\begin{equation}
    [x,y]=2z, \quad [x,z]=-2xy+2z+\ga, \quad [y,z]=y^2+P(x)-n, 
\end{equation}
\begin{equation}
    Q(x)+x(y^2-n)+z^2-2yz-\ga y=0,
\end{equation}
where $P(x)$ is the unique degree $n-2$ polynomial satisfying
\begin{equation*}
    Q(-x(x-1))-Q(-x(x+1))=(x-1)P(-x(x-1))+(x+1)P(-x(x+1)).
\end{equation*}

\begin{theorem}[{\cite[Theorem 2.22 \& 3.6]{levyIsomorphismProblemsNoncommutative2009} }]\label{teo iso quantizations type D}

    For all $n\geq 4$, there is an isomorphism $\Sigma:D_n(Q,\ga)\rightarrow D_n(Q,-\ga)$ given by $\Sigma:=(x,-y,-z)$.    
    If $n=4$, write $Q(x)=x^3+ax^2+bx+c$. Set 
    \begin{equation*}
    \begin{aligned}
        &b'=\frac{1}{8}(3a^2-4b-12 i \gamma), \\
        &c'=c+\frac{1}{16}(a^3-4ab-4ia\gamma), \\
        &\gamma'=\frac{i}{8}(a^2-4b+4i\gamma),
    \end{aligned}
    \end{equation*}
    and let $Q'(x)=x^3+ax^2+b'x+c'$. There is an isomorphism $T:D_n(Q,\ga)\rightarrow D_n(Q',\ga')$ given by
    \begin{equation*}
        T:=\left(-\frac{1}{2}x+\frac{i}{2}y-\Big(1+\frac{1}{4}a\Big),\frac{3i}{2}x-\frac{1}{2}y+i\Big(1+\frac{1}{4}a\Big),z\right),
    \end{equation*}
    \begin{enumerate}[(i)]
        \item $\Sigma$ and $T$ satisfy the $S_3$ relations, whenever the compositions make sense. In particular $\Sigma^2=\id$ and $T^3=\id$;
        \item The only possible isomorphisms between quantizations $D(Q,\ga)$ are $\id, \Sigma, T $, $ T^{-1}, (\Sigma\circ T) $ and $ (\Sigma\circ T^{-1})$.
    \end{enumerate}
    In particular, the group $\Aut(D_n(Q,\gamma))$ has the following description. If $n>4$:
\begin{itemize}
    \item if $\gamma\neq0$, then $\Aut(D_n(Q,\gamma))=\langle\id\rangle$;
    \item if $\gamma=0$, then $\Aut(D_n(Q,\gamma))=\ZZ/2\ZZ$, generated by $\Sigma$.
\end{itemize}
If $n=4$:
\begin{itemize}
    \item if $\gamma=0$ and $b=a^2/4$, then $\Aut(D_4(Q,\gamma))=S_3$, generated by $T$ and $\Sigma$;
    \item if $\gamma\neq0$ and $b=a^2/4-i\gamma$ (respectively $b=a^2/4+i\gamma$), then $\Aut(D_4)=\ZZ/2\ZZ$, generated by $\Sigma\circ T$ (respectively $\Sigma\circ T^{-1}$);
    \item if $\gamma=0$ and $b\neq a^2/4$, then $\Aut(D_4(Q,\gamma))=\ZZ/2\ZZ$, but this time it is generated by $\Sigma$;
    \item if $\gamma\neq0$ and $b\neq a^2/4\pm i\gamma$, then $\Aut(D_4(Q,\gamma))=\langle\id\rangle$.
\end{itemize}
\end{theorem}

\begin{oss}
    Our parametrization of the quantizations $D_n(Q,\ga)$ differs slightly from Levy's. The algebra defined in \cite[Definition 1.5]{levyIsomorphismProblemsNoncommutative2009} coincides with our $D_n(Q+nx,\ga)$.
\end{oss}

\subsection{Isomorphisms of deformations}

In this section we compute all the Poisson isomorphisms between the deformations $\Dd_n(Q,\gamma)$ and confirm Conjecture \ref{conj main}.

Let us first compute the groupoid $\Iso(\Dd_n)$ of isomorphisms between the deformations $\Dd_n(Q,\gamma)$ as affine varieties. We follow a similar method to \cite{blancNonrationalityFibrationsAssociated2015}. We embed each $\Dd_n(Q,\gamma)$ into a projective normal surface $X_n(Q,\gamma)$, such that every point in $X_n(Q,\gamma)\setminus\Dd_n(Q,\gamma)$ is smooth in $X_n(Q,\gamma)$. We then prove that every isomorphism $\Dd_n(Q,\gamma)\rightarrow\Dd_n(Q',\gamma')$ extends to an isomorphism of $X_n(Q,\gamma)\rightarrow X_n(Q',\gamma')$. This way we reduce ourselves to study the groupoid $\Iso(X_n, \Dd_n)$ of isomorphisms $X_n(Q,\gamma)\rightarrow X_n(Q',\gamma')$ that restrict to $\Dd_n(Q,\gamma)\rightarrow\Dd_n(Q',\gamma')$, which is simpler to compute.  

We construct $X_n(Q,\gamma)$ as a hypersurfaces of $F_{a,b}$, a $\PP^2$-bundle
over $\PP^1$. This $\PP^2$-bundle is $\PP(\Oo_{\PP^1}\oplus\Oo_{\PP^1} (a)\oplus
\Oo_{\PP^1} (b))$, and can be viewed as the gluing of $U_{a,b,0} = \PP^2\times\CC$ and $U_{a,b,\infty} = \PP^2 \times\CC$ along $\PP^2 \times\CC^*$, where the identification map is given by the involution
$$((w : y : z), x)\longrightarrow \left((w : x^{-a}y : x^{-b}z), \ \frac{1}{x}\right).$$
The $\PP^2$-bundle is
given by the map $F_{a,b}\rightarrow\PP^1$ corresponding to $((w : y : z), x) \mapsto (x : 1)$ in the first
chart and $((w : y : z), x) \mapsto (1 : x)$ in the second one.
If $n = 2k$, we take the $\PP^2$-bundle $F_{k-1,k-1}$,
and denote by $X_n=X_n(Q,\gamma)$ the projective surface that restricts to the following surfaces on each
chart:
\begin{align*}
    &\{((w : y : z), x) \in U_{k-1,k-1,0} \, | \,Q(x)w^2 + xy^2 + z^2-\gamma yw=0\},\\
&\{((w : y : z), x) \in U_{k-1,k-1,\infty} \,| \, Q^{r}(x)w^2 + y^2 + xz^2-\gamma x^{k}yw=0\}.
\end{align*}
Here by $Q^{r}(x)$ we mean the reciprocal polynomial of $Q$, i.e. $Q^r=x^{n-1}Q(x^{-1})$.

If $n = 2k + 1$, we take the $\PP^2$-bundle $F_{k-1,k}$,
and denote by $X_n=X_n(Q,\gamma)$ the projective surface that restrict to the following surfaces on each
chart:
\begin{align*}
    &\{((w : y : z), x) \in U_{k-1,k,0}\, | \, Q(x)w^2 + xy^2 + z^2-\gamma yw=0\},\\
&\{((w : y : z), x) \in U_{k-1,k,\infty} \, | \, Q^{r}(x)w^2 + xy^2 + z^2-\gamma x^{k+1}yw=0\}.
\end{align*}

In both cases, we embed the surface $\Dd_n$ in the first affine chart of $X_n$, via the embedding
$$(x, y, z)\mapsto ((1 : y : z), x),$$
so that $X_n$ is the closure of $\Dd_n$ in $F_{a,b}$. 
Geometrically the situation is similar to the case considered in \cite{blancNonrationalityFibrationsAssociated2015} and we can make similar remarks. All the singular points of $X_n$ are in the image, under the above embedding, of a singular point in $\Dd_n$; in particular, for generic $Q$ and $\gamma$, the surface $\Dd_n$ is smooth and so is $X_n$.

The $\PP^2$-bundles $F_{k-1,k-1}\rightarrow \PP^1$ and $F_{k-1,k}\rightarrow \PP^1$ restrict to a morphism $\rho : X_n \rightarrow \PP^1$. The fibres are conics in $\PP^2$, which are smooth for generic values of $x$. The generic fibre is thus isomorphic to $\PP^1$. The fibre $F_\infty$ over $(1 : 0)$ is always degenerate. It is a union of two transversal lines in the second chart, given by the equations $x=0, \, w=\pm iy$ (if $n$ is even) or  $x=0, \, w=\pm iz$ (if $n$ is odd). Denote the two lines in $F_\infty$ by $F_+$ and $F_-$, respectively. The other degenerate fibres are the ones lying over the solutions of the equation
$$\det \begin{pmatrix}
Q(x) & -\gamma/2 & 0\\
-\gamma/2 & x & 0 \\
0 & 0 & 1\\

\end{pmatrix}=0.	 $$
Counting with multiplicity, we have $n+1$ points in $\PP^1$ with degenerate fibre.

For every $Q$ and $\gamma$, the complement $X_n(Q,\gamma)\setminus\Dd_n(Q,\gamma)$ consists of the curve $C_n$, given by the equation $w = 0 $ in each chart,
and the curve $F_\infty=F_+\cup F_- $, with equation $x = 0$ in the second chart, corresponding to the fibre over $(1:0)$. We do not specify the deformation parameters when referring to the curves $C_n, F_+, F_-$, since they are defined by the same equation in $F_{k-1,k-1}$ or $F_{k-1,k}$. 

The geometric description of the boundary is the same for every deformation parameter. We have the following (compare with \cite[Lemma 4.3]{blancNonrationalityFibrationsAssociated2015}).

\begin{lemma}
    For every monic polynomial $Q$ of degree $n-1$ and for all $\gamma\in\CC$, the complement of $\Dd_n(Q,\gamma)$ in $X_n(Q,\gamma)$ is the union of the three curves $C_n, F_+, F_-$, all isomorphic to $\PP^1$. Any two of them intersect transversally, in exactly one point, which is $C_n\cap F_+\cap F_-$. Moreover, $C_n^2=3-n$, and $F_+^2=F_-^2=-1$. 
\end{lemma}
\begin{proof}
    The only thing that doesn't follow from the discussion above are the self intersection numbers. Let $P_n$ be the curve given by the equation $y=0$ in both charts. If $n=2k$, $P_n$ and $C_n$ intersect only in the point $((0:0:1),0)$ of the second chart. This happens along the distinct directions $w=0$ and $y=0$, so $C_n\cdot P_n=1$. If $n=2k+1$, the two curves are disjoint, so $C_n\cdot P_n=0$. In both cases, $C_n\cdot P_n=2k+1-n$.

    Consider the rational map $g\in\CC(X_n)^*$ given by $w/y$ on the second chart, and by $wx^{k-1}/y$ on the first chart. The associated principal divisor is $C_n $ $ +(k-1)F_0-P_n$, where $F_0$ is given by the equation $x=0$ in the first chart, i.e. it is the fibre over $(0:1)$. Computing the principal divisor associated to the rational function $x$, it is clear that $F_0$ is linearly equivalent to $F_\infty$. The intersection of $C_n$ and $F_\infty$ has multiplicity 2, since $$C_n\cdot F_\infty=C_n\cdot(F_++ F_-)=2. $$
    Thus
    $$C_n^2=C_n\cdot(P_n-(k-1)F_{\infty})=(2k+1-n)-(2k-2)=3-n. $$
    Clearly, $F_0$ and $F_\infty$ are disjoint, so $F_0\cdot F_\infty=0$. Since they are linearly equivalent, this means 
    $$0=F_0\cdot F_\infty=F_\infty^2=(F_++F_-)^2=F_+^2+F_-^2+2F_+\cdot F_-=F_+^2+F_-^2+2. $$
    The linear equivalence of $F_0$ and $F_\infty$ implies $F_+=F_0-F_-$, so $F_+^2=F_-^2+F_0^2.$ Since $F_0^2=F_0\cdot F_\infty=0$, this implies that $F_+^2=F_-^2=-1$.
\end{proof}

Consider now $\phi\in\Iso(\Dd_n)$, an isomorphism $\Dd_n(Q,\gamma)\rightarrow\Dd_n(Q',\gamma')$. This extends to a birational map $\phi:X_n(Q,\gamma)\dashrightarrow X_n(Q',\gamma')$, that is biregular between $\Dd_n(Q,\gamma)\subset X_n(Q, \gamma)$ and $\Dd_n(Q',\gamma')\subset X_n(Q', \gamma')$. We recall the following results about birational maps and blow-ups.

\begin{lemma}\label{lemma birational map blow-ups diagram}
    Let $X, X'$ be projective complex surfaces and $\phi:X\dashrightarrow X'$ a birational map. Assume that $X$ and $X'$ are smooth outside of the open sets $S$ and $S'$, and that $\phi:S\rightarrow S'$ is biregular. Then there exists a surface $Z$ and a commutative diagram
    \begin{equation}
        \xymatrix{
&Z\ar[ld]_\eta \ar[rd]^\pi \\
X\ar@{.>}[rr]_\phi & &X'        
        }
    \end{equation}
    where the morphisms $\eta,\pi$ are composite of blow-ups.
\end{lemma}
\begin{proof}
    If $X,X'$ are smooth, this is Theorem II.11 of \cite{beauvilleComplexAlgebraicSurfaces1996}. If $X$ or $X'$ are not smooth, $\phi$ induces a birational map $\hat{\phi}$ between resolutions $\hat{X}$ and $\hat{X'}$. We are now in the smooth case, so we get the following commutative diagram
    \begin{equation*}
         \xymatrix{
&Z'\ar[ld]_{\eta'} \ar[rd]^{\pi'} \\
\hat{X}\ar@{.>}[rr]_{\hat{\phi}} \ar[d]_f & &\hat{X'}\ar[d]^g \\
X\ar@{.>}[rr]_\phi & &X'        
        }
    \end{equation*}
    where both $f$ and $g$ are composition of blow-ups. Now, the singularities of $X$ and $X'$ are inside $S$ and $S'$ respectively, and $\phi$ is biregular between $S$ and $S'$. We can thus contract all the curves blown up when resolving the singularities. Call $Z$ the image of $Z'$ under these contractions. We end up with the following diagram
    \begin{equation}
        \xymatrix{
&Z\ar[ld]_\eta \ar[rd]^\pi \\
X\ar@{.>}[rr]_\phi & &X'        
        }
    \end{equation}
    with $\eta=f\circ\eta'\circ f^{-1}$ and $\pi=g\circ\pi'\circ g^{-1}$.
\end{proof}

\begin{lemma}[{\cite[Proposition II.2-3]{beauvilleComplexAlgebraicSurfaces1996}}] \label{lemma technical properties blow-ups}
   Let $S$ be a smooth surface, $\pi:\hat{S}\rightarrow S$ the blow-up of a point $p\in S$ and $E\subset \hat{S}$ the exceptional divisor. Then
   \begin{enumerate}[(i)]
       \item Let $C\subset S$ be an irreducible curve, then $\pi^*C=\overline{C}+mE$, where $m$ is the multiplicity of $p$ in $C$, and the curve $\overline{C}$ is the strict transform of $C$;
       \item Let $D,D'$ be divisors on $S$. Then $(\pi^*D)\cdot(\pi^*D')=D\cdot D'$, $E\cdot\pi^*D=0$ and $E^2=-1$.
   \end{enumerate}
\end{lemma}

\begin{cor}\label{cor technical properties blow-ups}
    In the setting of Lemma \ref{lemma technical properties blow-ups}, let $C\subset S$ be an irreducible curve. Then
    \begin{enumerate}[(i)]
        \item the intersection number $E\cdot\overline{C}$ is the multiplicity of $p$ in $C$.
        \item  $C^2\geq \overline{C}^2$, and $C^2=\overline{C}^2$ if and only if $p\notin C$;
    \end{enumerate}
\end{cor}
\begin{proof}
    From point (i) of Lemma \ref{lemma technical properties blow-ups}, $\pi^*C=\overline{C}+mE$. Multiply now both sides by $E$. Using point (ii) of Lemma \ref{lemma birational map blow-ups diagram}, this becomes $0=\overline{C}\cdot E-m$. 

    We compute the self intersection of $C$. Using again point (ii), we get 
    $$C^2=(\pi^*C)^2=(\overline{C}+mE)^2=\overline{C}^2+2m\overline{C}\cdot E-m^2E^2=\overline{C}^2+m^2\geq \overline{C}^2. $$
\end{proof}

We can apply Lemma \ref{lemma birational map blow-ups diagram} to our birational map $\phi$ and get the diagram
    \begin{equation}\label{diagram blow-ups}
        \xymatrix{
&Z\ar[ld]_\eta \ar[rd]^\pi \\
X_n(Q,\gamma)\ar@{.>}[rr]_\phi & &X_n(Q',\gamma')        
        }
    \end{equation}
    where the morphisms $\eta,\pi$ are compositions of blow-ups of points of $X_n(Q,\gamma)\setminus \Dd_n(Q,\gamma)$ and $X_n(Q',\gamma')\setminus\Dd_n(Q',\gamma')$ respectively. Denote by $\overline{F_+},\overline{F_-}$ and $\overline{C_n}$ the strict transforms via $\eta^{-1}$ of $F_+,F_-$ and $C_n$. Note that since we are blowing-up away from the singularities, we can assume to be in the setting of Lemma \ref{lemma technical properties blow-ups}. Since $\phi$ is biregular between $\Dd_n(Q,\gamma)$ and $\Dd_n(Q',\gamma')$, it sends the curves $F_+,F_-$ and $C_n$ into themselves. Suppose now that $\phi$ is not an isomorphism $X_n(Q,\gamma)\xrightarrow{\sim}X_n(Q',\gamma')$. Then $\pi$ must contract one of the curves $\overline{F_+},\overline{F_-}$ and $\overline{C_n}$ to a point. 

\begin{oss}\label{oss order blow-ups}
    Without loss of generality, we can assume that $\pi$ decomposes into a sequence of blow-ups of points $\pi_s\circ\cdots\circ\pi_1$, where $\pi_1$ contracts one of the curves $\overline{F_+},\overline{F_-}$ or $\overline{C_n}$. To see why, let us decompose $\eta=\eta_k\circ\dotso\circ\eta_1$ into blow-ups of points. Each $\eta_i$ adds an exceptional divisor, which we denote $E_i$. Denote by $\overline{E_i}$ their strict transforms. If $\pi_1$ contracts $E_k$, then we can remove $\pi_1$ and $\eta_k$ and get another diagram of the form \eqref{diagram blow-ups}. Suppose now that $\pi_1$ contracts $\overline{E_i}$ for some $i<k$. This means that $\overline{E_i}^2=-1$, so the self intersection number of $E_i$ does not change under the strict transform. By point (ii) of Corollary \ref{cor technical properties blow-ups}, the maps $\eta_j$ with $j>i$ blow-up points that are not in $E_i$. So $\eta_i$ commutes with all $\eta_j$ with $j>i$ and we reduce ourselves to the case $E_k$.
\end{oss}

We have the following proposition (compare with \cite[Proposition 4.4]{blancNonrationalityFibrationsAssociated2015}). 

\begin{prop}\label{Prop Blackbox}
    Let $n\geq4$, $Q,Q'$ monic polynomials of degree $n-1$ and $\gamma,\gamma'\in\CC$. Every isomorphism $\phi:\Dd_n(Q,\gamma)\rightarrow\Dd_n(Q',\gamma')$ extends to an isomorphism $X_n(Q,\gamma)\rightarrow X_n(Q',\gamma')$.
\end{prop}

\begin{proof}    
    Suppose that $\phi:\Dd_n(Q,\gamma)\rightarrow\Dd_n(Q',\gamma')$ does not extend to an isomorphism $X_n(Q,\gamma)\rightarrow X_n(Q',\gamma')$. We are in the situation described in diagram \eqref{diagram blow-ups}. By Remark \ref{oss order blow-ups}, we can assume that the first curve contracted by $\pi$ is either $\overline{F_+},\overline{F_-}$ or $\overline{C_n}$. Since this curve is a $(-1)$-curve, it is either $\overline{F_+} $ or $\overline{F_-}$ (or $\overline{C_n}$ if $n=4$), by (ii) of Corollary \ref{cor technical properties blow-ups}; say it is $\overline{F_+}$. 
    
    Since taking the strict transform via $\eta^{-1}$ does not change the self intersection number of $F_+$, $\eta$ does not blow-up any point of $F_+$ by (ii) of Corollary \ref{cor technical properties blow-ups}. In particular, it does not blow-up the triple intersection point of the components of the boundary. So $\overline{F_+},\overline{F_-}$ and $\overline{C_n}$ still intersect transversely in one point inside $Z$. Denote $F^{(1)}:=\pi_1(\overline{F_-})$ and $C^{(1)}=\pi_1(\overline{C_n})$. Then 
    $$F^{(1)}\cdot C^{(1)}=(\overline{F_-}+\overline{F_+})\cdot (\overline{C_n}+\overline{F_+})=1+1+1-1=2, $$
    by applying both points of Lemma \ref{lemma technical properties blow-ups} and the fact that $m=\overline{F_-}\cdot\overline{F_+}=1$ (by (i) of Corollary \ref{cor technical properties blow-ups}). Thus $F^{(1)}$ and $C^{(1)}$ are tangent. This leads to a contradiction, because it is not possible to recover the original boundary via other contractions. In fact, if $\pi$ contracts $F^{(1)}$ or $C^{(1)}$, say $F^{(1)}$, then $\pi(C^{(1)})$ is a curve whose strict transform $C^{(1)}$ intersects the exceptional divisor with multiplicity $2$. This means, by (i) of Corollary \ref{cor technical properties blow-ups}, that it contains a point with multiplicity $2$, so it is not smooth. This is a contradiction, because the curves of the boundary are all smooth, and contracting does not resolve singularities. If instead $\pi$ never contracts $F^{(1)}$ or $C^{(1)}$, then it does not modify their intersection point, and it is impossible to recover the original boundary. The same reasoning works if the curve contracted by $\pi_1$ is $\overline{F_-}$ (or $\overline{C_n}$ if $n=4$). 
\end{proof}
Using Proposition \ref{Prop Blackbox}, we can compute the groupoid $\Iso(\Dd_n)$ for all $n\geq4$ (compare with \cite[Corollary 4.5]{blancNonrationalityFibrationsAssociated2015}). First, let us recall some facts that will be used in the proof.

\begin{lemma}\label{Prop iso conic}
    Every isomorphism between two non-degenerate conics in $\PP^2$ can be extended to a projective transformation of $\PP^2$.
\end{lemma}
\begin{proof}
    Since two non degenerate conics are always isomorphic via a projective transformation of $\PP^2$, it is sufficient to show that every automorphism of a specific non degenerate conic extends to a projective transformation of $\PP^2$.

    Consider the conic $C: xy-z^2=0$. Since it is non degenerate, it is isomorphic to $\PP^1$. Explicitily, this is given by
    $$ (s:t)\mapsto (s^2:st:t^2).$$
    Thus, every automorphism of $C$ will be induced by an automorphism of $\PP^1$, which has the form $(s:t)\mapsto (as+bt:cs+dt)$, with $a,b,c,d\in\CC$ with $ad-bc\neq 0$. The induced automorphism on the conic is
    $$ (x:y:z)\mapsto (a^2x+2aby+b^2z: acx+(ad+bc)y+bdz: c^2+2cdy+d^2z),$$
    which extends to an element of $PGL(3,\CC)$, because the determinant of the associated matrix is \\$(ad-bc)^3$.
\end{proof}

\begin{prop}\label{Prop Makar-Limanov}
    Consider the algebra $A:=\CC[x,y,z]/(xy^2-z^2)$, graded with $\deg x=0, \deg y=1, \deg z=1$. The group of graded automorphisms of $A$ is:
    $$\{(\alpha^2\beta^{-2} x, \beta y, \alpha z) \, | \, \alpha,\beta \in \CC^*\} . $$ 
\end{prop}
\begin{proof}
    It is shown in \cite[Theorem 1]{makar-limanovGroupAutomorphismsSurface2001} that the group of automorphisms of $A$ is generated by:
    \begin{enumerate}
        \item Hyperbolic rotations: $H_{\nu}=(\nu^{-2} x,\nu y, z)$, for all $\nu\in\CC^*$;
        \item Rescalings: $R_\mu=(\mu^2 x, y, \mu z)$, for all $\mu\in\CC^*$:
        \item Triangular automorphisms: $$\Delta_g=(x+[(z+y^2g(y))^2-z^2]y^{-2}, y, z+y^2 g(y)),$$
        for all $g(y)\in\CC[y]$.
    \end{enumerate}
    The triangular automorphisms form a normal subgroup $\Delta$ isomorphic to the additive group $\CC[x]$ via $g\mapsto\Delta_g$, and the group of automorphisms is the semidirect product of $\Delta$ and of the subgroup generated by the automorphisms of type (1) and (2)  (\cite[Final Remark]{makar-limanovGroupAutomorphismsSurface2001}). The proposition follows by noticing that automorphisms in $\Delta$ don't preserve the grading (except for the identity), and that $H_\mu H_{\nu}=H_{\mu\nu}$, $R_\mu R_\nu=R_{\mu \nu}$ and $R_\mu H_\nu= H_{\nu} R_\mu$.    
\end{proof}

\begin{lemma}\label{prop bundle isomorphism}
    Suppose $\rho_1: X_1\rightarrow \PP^1$ and $\rho_2: X_2\rightarrow \PP^1$ are proper, surjective morphisms with irreducible fibres, and let $\phi$ be an isomorphism $X_1\rightarrow X_2$. If there exists a fibre in $X_1$ that $\phi$ sends to a fibre in $X_2$, then $\phi$ sends all fibres to fibres.
\end{lemma}
\begin{proof}
    Suppose there exists $x_0\in\PP^1$ such that $\phi(\rho_1^{-1}(x_0))=\rho_2^{-1}(x_1)$ for some $x_1\in\PP^1$. Take any other fibre $\rho_1^{-1}(z)$, with $z\in\PP^1$. Assume $\phi(\rho_1^{-1}(z))$ is not a fibre. Then $\rho_2(\phi(\rho_1^{-1}(z)))=\PP^1$ since $\rho_2$ is proper, which means that there is a point in $\phi(\rho_1^{-1}(z))$ that gets mapped to $x_1$. So, $\phi(\rho_1^{-1}(x_0))\cap \phi(\rho_1^{-1}(z))\neq\emptyset$, which is absurd because $\phi$ is an isomorphism.
\end{proof}

Define the following morphisms in $\Iso(\Dd_n)$, for $n\geq 4$.
\begin{enumerate}
    \item $$R^\pm_\l:\Dd_n(Q,\gamma)\rightarrow\Dd_n(Q',\gamma'),$$
    $$R^\pm_\lambda=(\lambda^2 x, \mu\lambda^{-2} y, \pm\mu\lambda^{-1} z), $$
    with $\mu=\gamma/\gamma'$, for all $Q,Q'$ such that $Q(\lambda^2 x)=\lambda^{2(n-1)}Q'(x)$ and for all $\gamma,\gamma'\neq0$ such that $\lambda^n=\pm\mu$;
    \item 
    $$P^\pm_\l:\Dd_n(Q,0)\rightarrow\Dd_n(Q',0),$$
$$P^\pm_\lambda=(\lambda^2 x, \pm\lambda^{n-2} y, \pm\lambda^{n-1} z), $$
for all $Q,Q'$ such that $Q(\lambda^2 x)=\lambda^{2(n-1)}Q'(x)$ and for all $\l\in\CC^\times$.
\end{enumerate}

\begin{theorem}\label{teo iso Dn}
Let $Q,Q'$ monic polynomials of degree $n-1$ and $\gamma,\gamma'\in \CC$.
\begin{enumerate}[(i)]
    \item If $n>4$, the only isomorphisms in $\Iso(\Dd_n)$ are of the form $R_\l^\pm$ and $P_\l^\pm$.
    \item If $n=4$, let $Q(x)=x^3+ax^2+bx+c$ and $Q'(x)=x^3+a'x^2+b'x+c'$ . Then we have an isomorphism $\tau:\Dd_n(Q,\gamma)\rightarrow\Dd_n(Q',\gamma')$ given by
\begin{equation}\label{def tau}
    \tau=\left(-\frac{1}{2}x+\frac{i}{2}y-\frac{1}{4}a,\frac{3i}{2}x-\frac{1}{2}y+\frac{i}{4}a,z\right),
\end{equation}
every time $(Q',\gamma')$ satisfies the following condition:

\begin{equation}\label{eq condition iso order 3}
    \begin{aligned}
        &a'=a, \\
        &b'=\frac{1}{8}(3a^2-4b-12 i \gamma), \\
        &c'=c+\frac{1}{16}(a^3-4ab-4ia\gamma), \\
        &\gamma'=\frac{i}{8}(a^2-4b+4i\gamma).
    \end{aligned}
\end{equation}
The groupoid $\Iso(\Dd_4)$ is generated by all isomorphisms of the form $R_\l^\pm$, $P_\l^\pm$ and $\tau$.

\end{enumerate}
\end{theorem}

\begin{proof}
It is easy to check that the listed actions are well defined isomorphisms. Thus, we only need to prove that they generate the whole of $\Iso(\Dd_n)$. We know from Proposition \ref{Prop Blackbox} that every element of $\Iso(\Dd_n)$ lifts to an element of $\Iso(X_n)$. We can thus consider the groupoid $K:=\Iso(X_n, \Dd_n)$ of isomorphisms $X_n(Q,\gamma)\rightarrow X_n(Q',\gamma')$ that restrict to $\Dd_n(Q,\gamma)\rightarrow\Dd_n(Q',\gamma')$.

First, let us compute the subgroupoid $K_0\leq K$ of isomorphisms that preserve the conic bundle structure, i.e. that send a fibre to another fibre. Every isomorphism $g\in K_0$ induces an automorphism of $\PP^1$. In addition, it needs to preserve the boundary $F_\infty \cup C_n$, so it must preserve the fibre over $(1:0)$. Thus, the action on $\PP^1$ is given by a degree one polynomial $x\mapsto ax+b$. Since $g$ sends conics to isomorphic conics, by Lemma \ref{Prop iso conic}, it extends to a projective transformation of the projective plane that contains them (which is the fibre of the $\PP^2$ bundle $F_{a,b}$). Since it also preserves the intersection of the conics with $C_n$ (i.e. with $w=0$) we can write $g$ (as an isomorphism $\Dd_n(Q,\gamma)\rightarrow\Dd_n(Q',\gamma')$) as
$$ (x,y,z)\mapsto (ax+b, cy+dz+e, fy+hz+k),$$
with $a,b\in\CC$ and $c,d,e,f,h,k\in\CC[x]$ satisfying $ch-df\neq0$. 

Let us consider on $\Dd_n(Q,\gamma)$ the filtration $\Ff$ induced by $\deg x=0$ and $\deg y=\deg z=1$. Then
$$ \gr_\Ff (\Dd_n(Q,\gamma))\cong A=\CC[x,y,z]/(xy^2-z^2)$$
for every deformation parameter $(Q,\gamma)$. Since $g$ preserves the filtration $\Ff$, it induces a graded automorphism $\gr (g)$ of $A$. By Proposition \ref{Prop Makar-Limanov}, $\gr (g)=(\alpha^2\beta^{-2} x, \beta y, \alpha z)$ for some $\alpha,\beta\in\CC^*$, so 
\begin{equation}\label{eq isomorphism g}
    g=(\alpha^2\beta^{-2} x + b, \beta y+c, \alpha z+d),
\end{equation}
with $b\in\CC$ and $c,d\in\CC[x]$. Plugging \eqref{eq isomorphism g} into the defining equation of $\Dd_n(Q,\gamma)$ we get the following condition:
$$Q(\alpha^2\beta^{-2} x + b)+(\alpha^2\beta^{-2} x + b)(\beta y+c)^2+(\alpha z+d)^2-\gamma (\beta y+c)=  k(Q'(x)+xy^2+z^2-\gamma' y), $$
for some $k\in\CC[x,y,z]$. Comparing the coefficients of $z^2,z,y^2,y$ and the constant term we get, respectively
\begin{align}
    &\alpha^2=k, \label{z^2}\\
    &2\alpha d=0,\label{z}\\
    &(\alpha^2\beta^{-2}x+b)\beta^2=kx, \label{y^2}\\
    &2(\alpha^2\beta^{-2}x+b)\beta c-\gamma\beta=-k\gamma', \label{y}\\
    &Q(\alpha^2\beta^{-2}x+b)=kQ'(x). \label{cost}    
\end{align}
From \eqref{z^2}, \eqref{z} and \eqref{y^2} we get that $k\in\CC^*$, $d=0$ and $b=0$. From \eqref{y}, since $k$ is a constant, we get that the coefficient of $x$ is $0$, so $c=0$. Combining \eqref{y},\eqref{z^2} and \eqref{y^2} we also get $\gamma=\gamma'=0 \wedge \gamma\beta=\alpha^2\gamma'$. From \eqref{cost} and \eqref{z^2} we have $Q(\alpha^2\beta^{-2} x)=\alpha^2 Q'(x)$. In particular, from comparing the $x^{n-1}$ coefficients, we get $\alpha^{2n-4}=\beta^{2n-2}$, so 
\begin{equation}\label{eq rel alpha beta}
    \alpha^{n-2}=\pm\beta^{n-1}.
\end{equation}
Consider $\lambda:=\alpha\beta^{-1}$. Then from \eqref{eq rel alpha beta} we get that $\lambda^{n-1}=\pm\alpha$ and $\lambda^{n-2}=\pm\beta$, with the same sign. We can thus rewrite the relation on $Q$ as 
\begin{equation}\label{eq Q Q'}
    Q(\lambda^2 x)=\lambda^{2(n-1)}Q'(x).
\end{equation} 
If we write $Q(x)=x^{n-1}+a_{n-2}x^{n-2}+\dots+a_0$, then by \eqref{eq Q Q'} we have $Q'(x)=x^{n-1}+\l^{-2}a_{n-2}x^{n-2}+\dots+\l^{-2(n-1)}a_0$.

Let us first consider the case with $\gamma,\gamma'\neq0$. Denote by $\mu:=\gamma/\gamma'$. We have the condition $\mu\beta=\alpha^2$. So $\pm\mu\lambda^{n-2}=\lambda^{2n-2}$, i.e. $\lambda^{n}=\pm\mu$. Thus, $K_0$ contains an isomorphism $R^\pm_\l:\Dd_n(Q,\gamma)\rightarrow\Dd_n(Q',\gamma')$ of the form
$$R^\pm_\lambda=(\lambda^2 x, \lambda^{n-2} y, \pm\lambda^{n-1} z)=(\lambda^2 x, \mu\lambda^{-2} y, \pm\mu\lambda^{-1} z), $$
for all $\lambda$ such that $\lambda^{n}=\pm\mu$ and for all $Q,Q'$ such that \eqref{eq Q Q'} holds.

When $\gamma=\gamma'=0$ we have no extra condition. Thus, $K_0$ contains an isomorphism $P^\pm_\l:\Dd_n(Q,0)\rightarrow\Dd_n(Q',0)$ of the form
$$P^\pm_\lambda=(\lambda^2 x, \pm\lambda^{n-2} y, \pm\lambda^{n-1} z), $$
for all $\lambda\in\CC^*$ and for all $Q,Q'$ such that \eqref{eq Q Q'} holds. By equations \eqref{z^2} to \eqref{cost}, the only elements in $K_0$ are of the form $R_\l, P_\l$.

As we noted before, the automorphisms in $K_0$ preserve the fibre $F_\infty$. On the other hand, if an isomorphism in $K$ preserves $F_\infty$, then it is in $K_0$, by Lemma \ref{prop bundle isomorphism}. Thus an isomorphism in $K\setminus K_0$ does not preserve $F_\infty$ or, equivalently, $C_n$. Since the self intersection numbers for $C_n, F_+$ and $ F_-$ are $3-n, -1, -1$ we see that $\Iso(\Dd_n)=K_0$ for $n>4$. For $n=4$, one directly checks that we have an isomorphism $\tau: \Dd_4(Q,\gamma)\rightarrow\Dd_4(Q',\gamma')$ of order $3$ defined by \eqref{def tau}, with $Q',\gamma'$ as in \eqref{eq condition iso order 3}. The isomorphism $\tau$ does not preserve the conic bundle, so it cyclically permutes the curves $F_+,F_-$ and $C_4$. Together with the isomorphism $\sigma_y:\Dd_4(Q,\gamma)\rightarrow\Dd_4(Q,-\gamma)$ defined by $\sigma_y=(x,-y,z)$, we have a full action of $S_3$ on set of the three curves $F_+, F_-, C_4$.

Take now any isomorphism $\phi:\Dd_4(Q,\gamma)\rightarrow\Dd_4(Q',\gamma')$. If $\phi$ fixes the curves $F_+,F_-$ and $C_4$, then it is in $K_0$. Otherwise, we can compose $\phi$ with the appropriate permutation generated by $\tau$ and $\sigma_y$ to get an isomorphism that fixes the three curves, i.e. it is again in $K_0$. Since $\sigma_y\in K_0$, it follows that $\tau$ and $K_0$ generate the whole groupoid $K$.

\end{proof}

\begin{theorem}\label{teo Poisson iso Dn}
    Let $Q,Q'$ monic polynomials of degree $n-1$ and $\gamma,\gamma'\in \CC$. 
    \begin{itemize}
        \item If $n>4$ there are only two classes of Poisson isomorphisms, $\id$ and $\sigma$, where 
        \begin{equation}\label{def sigma}
            \sigma:\Dd_n(Q,\gamma)\rightarrow\Dd_n(Q,-\gamma), \hspace{8mm} \sigma=(x,-y,-z).
        \end{equation}
    \item If $n=4$ the groupoid $\PIso(\Dd_4)$ is generated by the isomorphism $\sigma$ defined in \eqref{def sigma} and by the isomorphism $\tau$ defined in \eqref{def tau}. Moreover, every time the composition makes sense, $\sigma$ and $\tau$ satisfy the $S_3$ relations $\tau^3=\id$, $\sigma^2=\id$ and $\sigma\circ\tau\circ\sigma=\tau^2$.
    \end{itemize}
\end{theorem}
\begin{proof}
    It is sufficient to check when the isomorphisms of Theorem \ref{teo iso Dn} satisfy the Poisson relations. In particular, by imposing $\{x,y\}=2z$, we get that $R_\l^\pm$ and $P_\l^\pm$ are Poisson isomorphisms only if $\lambda=1$, in which case $R_1^+=P_1^+=\id$ and $R_1^-=P_1^-=\sigma$. It is then straightforward to check that $\sigma$ is a Poisson isomorphism. For the case $n=4$, one can directly check that $\tau$ is a Poisson isomorphism and that $\tau$ and $\sigma$ satisfy the $S_3$ relations.
\end{proof}

Putting together Theorem \ref{teo iso quantizations type D} and \ref{teo Poisson iso Dn} we can confirm Conjecture \ref{conj main} for type $\mathbf{D}$. 
\begin{theorem}\label{teo iso type D}
    Let $n\geq 4$. We have an isomorphism of groupoids
    \begin{equation*}
        \Iso(D_n)\cong\PIso(\Dd_n).
    \end{equation*}
    In particular, for every deformation parameter $(Q,\ga)$,
    \begin{equation*}
        \Aut(D(Q,\ga))\cong\PAut(\Dd(Q,\ga)).
    \end{equation*}
\end{theorem}

We can also describe the groups of affine automorphisms of $\Dd_n(Q,\gamma)$ by checking for which deformation parameters the isomorphisms in Theorem \ref{teo iso Dn} are automorphisms. 

\begin{theorem}\label{teo aut Dn}
Let $n\geq4$, $Q$ monic of degree $n-1$ and $\gamma\in\CC$. Write $Q(x)=x^dP(x^m)$, with $d\in\ZZ_{\geq 0}$ and $m\in\ZZ_{\geq 1}$, with $m$ maximal possible.
\begin{itemize}
    \item If $\gamma\neq0$, the group 
    \begin{equation}
        G=\{(\lambda^2 x, \lambda^{-2} y, \pm\lambda^{-1} z) \, | \, \lambda^{2n}=1, \, \lambda^{2m}=1 \}
    \end{equation}
    acts on $\Dd_n(Q,\gamma)$ via automorphism.
    \item If $\gamma=0$, the group
    \begin{equation}
        G=\{(\lambda^2 x, \pm\lambda^{n-2} y, \pm\lambda^{n-1} z) \, | \, \lambda^{2m}=1 \}
    \end{equation}
    acts on $\Dd_n(Q,\gamma)$.
\end{itemize}
If $n>4$, then $\Aut(\Dd_n(Q,\gamma))=G$. 

In particular, for generic $(Q,\gamma)$, $$\Aut(\Dd_n(Q,\gamma))=\ZZ/2\ZZ,$$ generated by $\sigma_z=(x,y,-z)$. 

For generic $Q$ and $\gamma=0$, $$\Aut(\Dd_n(Q,\gamma))=(\ZZ/2\ZZ)\times(\ZZ/2\ZZ),$$ with $\sigma_y=(x,-y,z)$ added to the $G$.

If $n=4$, write $Q(x)=x^3+ax^2+bx+c$. Then:
\begin{itemize}
    \item if $\gamma=0$ and $b=a^2/4$, then $\Aut(\Dd_4(Q,\gamma))$ is generated by $G$ and $\tau$;
    \item if $\gamma\neq0$ and $b=a^2/4-i\gamma$ (respectively $b=a^2/4+i\gamma$), then $Aut(\Dd_4(Q,\gamma))$ is generated by $G$ and $\sigma\circ\tau$ (respectively $\sigma\circ\tau^{-1}$);
    \item otherwise, $Aut(\Dd_4(Q,\gamma))=G$.
\end{itemize}

\end{theorem}
\begin{proof}
    Theorem \ref{teo aut Dn} follows at once from Theorem \ref{teo iso Dn}, by noticing that $Q(\l^2 x)=\l^{2(n-1)}Q(x)$ if and only if $Q(x)=x^dP(x^m)$, with $\l^{2m}=1$. 
\end{proof}

\begin{oss}
    Notice that we recover from Theorem \ref{teo aut Dn} the result of \cite[Corollary 4.5]{blancNonrationalityFibrationsAssociated2015} for the undeformed case, i.e. when $Q(x)=x^{n-1}$ and $\gamma=0$.
\end{oss}

\subsection*{Final remarks}

Let $X$ be a conic symplectic singularity. The moduli space of filtered quantization and deformations of $X$ is the quotient of a Cartan space $\mathfrak{P}$ by the action of a Weyl group $W$ \cite{namikawaPoissonDeformationsAffine2010, losevDeformationsSymplecticSingularities2022}. All the filtered (Poisson) isomorphisms between deformations and quantizations are induced by a graded Poisson automorphism of $\CC[X]$ \cite[Proposition 3.21 and Corollary 3.22]{losevDeformationsSymplecticSingularities2022}. If $X=V/\Gamma$ is a symplectic quotient, then the group of graded Poisson automorphisms of $\CC[X]$ is equal to $\Theta:=N_{Sp(V)}(\Gamma)/\Gamma$ \cite[Lemma 3.20]{losevDeformationsSymplecticSingularities2022}. 

From the results of this paper, we know that, for Kleinian singularities of type $\mathbf{A}$ and $\mathbf{D}$, there are no ``wild'' non-filtered isomorphisms in $\Iso$ and $\PIso$. In fact, in type $\mathbf{D}$ there exist only filtered isomorphisms, while in type $\mathbf{A}$ the only non-filtered isomorphisms are inner, that is they are the exponentiation of an inner nilpotent derivation. 

For a Kleinian singularity $X$ with Dynkin diagram $\Delta$, the group of diagram automorphisms $\Aut(\Delta)$ acts on $\CC[X]$ by graded Poisson automorphisms, and can be identified with a subgroup of $\Theta$. In every type except type $\mathbf{A}$ the diagram automorphisms coincide with $\Theta$, while in type $\mathbf{A}$ they are a proper subgroup. In this case, the hyperbolic rotations are not in $\Aut(\Delta)$. Interestingly, the only isomorphisms in $\PIso(\Aa)$ that are not automorphisms come from $\Aut(\Delta)$. In other words, the action of the hyperbolic rotations and the inner automorphisms on the moduli space $\mathfrak{P}/W$ is trivial.

For type $\mathbf{E}$, we expect no non-filtered isomorphisms to exist. Thus $\PIso(\Ee)$ and $\Iso(E)$ should be trivial, except for type $\mathbf{E_6}$, where $\Aut(\Delta)=\ZZ/2\ZZ$. Unfortunately, an explicit presentations of the quantizations in type $\mathbf{E}$ with generators and relations is not known. 

It would be interesting to understand if the absence of ``wild'' non-filtered isomorphisms for Kleinian singularities is incidental or if Conjecture \ref{conj main} holds in the more general setting of conical symplectic singularities.

\printbibliography

\Addresses

\end{document}